\documentclass[a4paper, 10pt]{article}

\usepackage[hang,flushmargin]{footmisc}

\usepackage[dvipsnames]{xcolor}
\usepackage{graphicx}
\usepackage[a4paper,top=28mm,bottom=32mm,left=28mm,right=28mm]{geometry}

\usepackage{amsthm,amsmath,amsfonts,amssymb}
\usepackage{mathtools}

\usepackage{textpos}

\usepackage{caption,subcaption}
\usepackage{enumitem}

\usepackage{tikz, tikz-cd}

\usepackage{float}

\usepackage[hyperfootnotes=false,colorlinks,
linkcolor={blue!80!black},
citecolor={red!60!black},
urlcolor={blue!80!black}]{hyperref}

\numberwithin{equation}{section}

\newtheorem{theorem}{Theorem}[section]
\newtheorem{lemma}[theorem]{Lemma}
\newtheorem{proposition}[theorem]{Proposition}

\newtheorem{question}[theorem]{Open question}

\newtheorem{letterthm}{Theorem}
\newtheorem{lettercor}[letterthm]{Corollary}

\newtheorem{letterthmbody}{Theorem}
\newtheorem{lettercorbody}[letterthmbody]{Corollary}
\theoremstyle{definition}
\newtheorem{definition}[theorem]{Definition}
\newtheorem{example}[theorem]{Example}

\theoremstyle{remark}
\newtheorem*{remark}{Remark}

\makeatletter
\renewenvironment{proof}[1][\proofname]{%
	\par\pushQED{\qed}%
	\normalfont\topsep6\p@\@plus6\p@\relax
	\trivlist
	\item[\hskip\labelsep\bfseries #1\@addpunct{.}]%
	\ignorespaces
}{%
	\popQED\endtrivlist\@endpefalse
}
\makeatother

\newcommand\xqed[1]{%
	\leavevmode\unskip\penalty9999 \hbox{}\nobreak\hfill
	\quad\hbox{#1}}
\newcommand\demo{\xqed{$\triangle$}}

\newcommand{\CMP}{$\mathsf{CMP}$} 

\newcommand{\G}{\mathbf{G}}
\newcommand{\HH}{\mathbf{H}}
\newcommand{\U}{\mathbf{U}}

\newcommand{\T}{\mathbf{T}}

\begin{document}

\begin{center}
	{\boldmath\LARGE\bf Automorphisms and monomorphisms of direct products of virtually solvable minimax groups}
	
\vspace{2ex}
	
	{\sc Jonas Der\'e$^\text{1}$ and Ken Vandermeersch\footnote{\noindent KU Leuven Kulak, Kortrijk, Belgium.  Corresponding author: Jonas Der\'e: \href{mailto:jonas.dere@kuleuven.be}{\texttt{jonas.dere@kuleuven.be}}.}}
\end{center}

\begin{abstract}
	\noindent  This paper studies automorphisms and monomorphisms of direct products $\Gamma=\Gamma_1\times\cdots\times\Gamma_r$ of finitely generated virtually solvable minimax groups, a class containing all virtually polycyclic groups. Under an indecomposability assumption on the $\mathbb Q$-algebraic hulls, we prove that every monomorphism of $\Gamma$ factorizes uniquely as $\varphi=\theta\cdot\zeta$, where $\theta$ sends each factor into a permuted factor with $\mathbb Q$-isomorphic hull and $\zeta$ is central and off-diagonal. Conversely, every such pair defines a monomorphism of $\Gamma$, and \(\varphi\) is an automorphism if and only if \(\theta\) is. This indecomposability assumption is sharp: we show it cannot be weakened to direct indecomposability of the factors.
	
	The proof proceeds in three steps: first by establishing the corresponding central mixing property for finite-dimensional Lie algebras and algebraic Lie algebras, then for connected linear algebraic groups, and finally by transferring these results to minimax groups via $\mathbb Q$-algebraic hulls. This extends the previously known nilpotent case both from automorphisms to monomorphisms and from finitely generated torsion-free nilpotent groups to the broader class of finitely generated virtually solvable minimax groups. As applications, we characterize co-Hopfian direct products and derive formulas for Reidemeister numbers and Reidemeister spectra.
\end{abstract}

\section{Introduction}
Computing automorphism groups is notoriously subtle, even for groups built from simpler pieces. 
Let \[G = G_1 \times G_2 \times \cdots \times G_r\] be a direct product of directly indecomposable groups. We consider the following question: 
\medskip 

\noindent \textbf{Question.} \textit{Can one describe \(\operatorname{Aut}(G)\) in terms of the groups \(\operatorname{Aut}(G_i)\)?}

\medskip 

\noindent Of course, given \(\theta_i \in \operatorname{Aut}(G_i)\), we obtain a diagonal automorphism \(\theta  = \operatorname{Diag}(\theta_1,\theta_2,\ldots,\theta_r)\) of \(G\), but in general automorphisms may also \emph{mix} the factors. For simplicity, we restrict our attention to non-abelian factors. 

 In the case of non-abelian finite groups, any \emph{mixing of the factors} is known to be necessarily {central}.  
  In \cite{Bidwell1,Bidwell2} it is proved that if $G = G_1 \times G_2 \times \cdots \times G_r$ is a direct product of non-abelian, directly indecomposable finite groups, then every $\varphi \in \operatorname{Aut}(G)$ admits a unique factorization \[\varphi = \theta \cdot \zeta \colon \quad x \mapsto \theta(x) \cdot  \zeta(x),\] where:
	\begin{enumerate}[label={(\Alph*)}, ref={\Alph*}]
	\item $\theta \in \operatorname{Aut}(G)$ maps each factor $G_i$ onto an isomorphic  $G_{\sigma(i)}$ for some unique permutation $\sigma$ in the symmetric group $\mathcal S_r$ on $r$ letters, and \label{item:Introduction1}
	\item $\zeta \in \operatorname{Hom}(G,Z(G))$ is \emph{off-diagonal}, i.e.\ \(\zeta\) satisfies  $\pi_{\sigma(i)}(\zeta(G_i))=\{e\}$ for each $i$.\(\!\!\) \label{item:Introduction2}
\end{enumerate}
Conversely, any such pair $\theta,\zeta$ satisfying {\eqref{item:Introduction1}} and {\eqref{item:Introduction2}} above yields an automorphism $\varphi:=\theta\cdot\zeta$ of \(G\).  In particular, apart from permuting isomorphic factors, the only non-trivial mixing is \emph{central}. 

This motivates the following question: for which classes of groups is the only mixing of direct factors central, up to permutation of isomorphic factors?

\vspace*{-0.1cm}

\begin{definition}[Central Mixing Property]
	Let $G=G_1 \times G_2 \times \cdots\times G_r$ be a direct product of non-abelian, directly indecomposable groups. We say that $G$ satisfies the \emph{Central Mixing Property} (\CMP{}) on automorphisms if every $\varphi\in\operatorname{Aut}(G)$ factorizes as $\varphi=\theta\cdot\zeta$ with $\theta,\zeta$ as in {\eqref{item:Introduction1}} and {\eqref{item:Introduction2}}, and conversely any such pair $(\theta,\zeta)$ induces an automorphism $\varphi:= \theta\cdot\zeta$ of \(G\).
	
			We define \CMP{} analogously in the categories of Lie algebras and linear algebraic groups.
\end{definition}

		\CMP{} is known to hold for direct products of finite groups \cite{Bidwell1, Bidwell2}, centerless groups \cite{Joh83,Sen21}, right-angled Artin groups \cite[Proposition 3.3]{GW16}, and finitely generated torsion-free nilpotent groups with directly indecomposable rational Mal'cev completions \cite{Sen24}. Not all direct products satisfy \CMP{}, however: our examples show that both the forward and converse directions of \CMP{} already fail for directly indecomposable torsion-free nilpotent groups with decomposable rational Mal'cev completions (see Subsection~\ref{subsec:assumptionsDECOMP}). 

Nevertheless, these counterexamples suggest that the failure of \CMP{} is more subtle than the failure of a central factorization \(\varphi=\theta\cdot\zeta\) itself. Indeed, the automorphisms in Subsection~\ref{subsec:assumptionsDECOMP} are still of this form with \(\zeta\colon G\to Z(G)\) central; what breaks down is the stronger off-diagonal condition on \(\zeta\). This shows that, already in the infinite nilpotent setting, the component-map description of Bidwell and Senden does not capture the full picture.
\begin{question}
	Let \(\Gamma = \Gamma_1 \times \Gamma_2\times  \cdots \times \Gamma_r\) be a direct product of directly indecomposable finitely generated torsion-free nilpotent groups. 	
	Does every automorphism \(\varphi \in \operatorname{Aut}(\Gamma)\) factorize as \(
	\varphi = \theta \cdot \zeta,
	\)
	where \(\theta\) maps each factor \(\Gamma_i\) onto an isomorphic factor \(\Gamma_{\sigma(i)}\), and \(\zeta\colon \Gamma \to Z(\Gamma)\) is central?
\end{question}

\subsection*{Main results}

The central result of this paper is Theorem~\ref{thm:C-minimax}.  It extends \cite{Sen24} in two directions:  we work with a strictly larger class of groups, namely finitely generated virtually solvable minimax groups, and we prove \CMP{} for \emph{monomorphisms} rather than only automorphisms. 

The structure of the proof is as follows. First \CMP{} is established for finite-dimensional Lie algebras, \(K\)-algebraic Lie algebras, and a class of \(K\)-defined linear algebraic groups in Theorems~\ref{thm:A-KalgebraicLieAlgebras} and \ref{thm:B-LAG}. We then transfer these results to minimax groups to obtain Theorem~\ref{thm:C-minimax}, via the notion of the $\mathbb Q$-algebraic hull. Finally, in Theorems~\ref{thm:D-cohopfian}, \ref{thm:E-Reidemeister} and Corollary~\ref{cor:F-spectrum}, we derive applications to co-Hopfian groups and twisted conjugacy/Reidemeister theory in direct products of minimax groups.

\paragraph{Lie algebras and linear algebraic groups.} Our first theorem is a generalization of a result used in \cite[Theorem 3.4]{FGH13} over $\mathbb R$. It establishes the Central Mixing Property for finite-dimensional Lie algebras over arbitrary fields. 

\begin{theorem}\label{thm:A-abstractLieAlgebras}
	Let $\mathcal L = \mathcal L_1 \oplus \mathcal L_2 \oplus \ldots  \oplus \mathcal L_r$ be a direct sum of non-abelian, directly indecomposable  finite-dimensional Lie algebras over an arbitrary field. Then $\mathcal L$ satisfies \CMP{} on automorphisms. 
\end{theorem}

 An important step in proving Theorem~\ref{thm:A-abstractLieAlgebras} is a \emph{Krull-Schmidt theorem} for Lie algebras establishing uniqueness of a direct sum decomposition into indecomposable summands (Theorem~\ref{thm:KalgebraicKrullSchmidt}). This corresponds to \cite[Theorem 3.3]{FGH13}, but that theorem contains a crucial error if some of the factors are abelian, for which we provide a simple counterexample in Example~\ref{example:counterexampleFHS}.

We prove a novel \emph{\(K\)-algebraic} variant for both the Krull-Schmidt theorem and Theorem~\ref{thm:A-abstractLieAlgebras}, for the Lie algebras of linear algebraic groups defined over some subfield \(K\subset\mathbb C\). 
In this setting, we consider \emph{\(K\)-algebraic} Lie algebra endomorphisms, i.e.\ endomorphisms that arise as differentials of \(K\)-defined algebraic group endomorphisms. We say a \(K\)-defined linear algebraic group \(\G\) is \(K\)-indecomposable if \(\G\) cannot be written as a direct product of non-trivial \(K\)-defined Zariski closed  subgroups. Similarly, we say its Lie algebra \(\mathfrak g\) is \(K\)-algebraically indecomposable if it cannot be written as a direct sum of two Lie ideals that arise as the Lie algebras of \(K\)-defined Zariski closed subgroups of \(\G\).  See Section~\ref{sec:PreliminariesLAG} for our conventions concerning linear algebraic groups.  

Since the proof for abstract Lie algebras is simpler, over \(\mathbb R\) it is available in \cite{FGH13}, and the main conceptual steps are the same, we give a detailed proof only in the \(K\)-algebraic case, see Subsections~\ref{subsec:KAlgebraicKS} and~\ref{subsec:KalgebraicAutomorphismsOfADirectSumOfLieAlgebras}.

\begin{letterthm}\label{thm:A-KalgebraicLieAlgebras}
	Let $\G = \G_1 \times \G_2 \times \cdots \times \G_r$ be a direct product of Zariski connected non-abelian \(K\)-defined linear algebraic groups. Let $\mathfrak g=\mathfrak g_1\oplus \mathfrak g_2 \oplus \ldots \oplus\mathfrak g_r$ be its Lie algebra, and assume every summand \(\mathfrak g_i\) is \(K\)-algebraically indecomposable.  
	Then \(\mathfrak g\) satisfies \CMP{} on \(K\)-algebraic automorphisms.
\end{letterthm}

These statements may then be pushed to situations where automorphisms of Lie algebras completely determine automorphisms of groups.  
For instance, Theorem~\ref{thm:A-abstractLieAlgebras} implies that a direct product of directly indecomposable, non-abelian, \(1\)-connected Lie groups satisfies the Central Mixing Property on automorphisms.

In the Zariski connected case, and under the additional hypothesis that $Z(\G)(\mathbb C)$ is torsion-free (intuitively corresponding to \(1\)-connectedness of Lie groups), statements about \(K\)-algebraic automorphisms of \(\mathfrak g\) lift to statements about \(K\)-automorphisms of $\G$.  
Using Theorem~\ref{thm:A-KalgebraicLieAlgebras} we will deduce:

\begin{letterthm}\label{thm:B-LAG}
	Let $\G = \G_1 \times \G_2 \times  \cdots \times \G_r$ be a direct product of Zariski connected non-abelian \(K\)-defined linear algebraic groups that are \(K\)-indecomposable and have torsion-free centers \(Z(\G_i)(\mathbb{C})\). 
	Then $\G$ satisfies \CMP{} on \(K\)-automorphisms.
\end{letterthm}

\paragraph{Minimax groups.}  Next, we consider direct products of finitely generated virtually solvable \emph{minimax} groups. This is a broad class of groups including all virtually polycyclic groups, but also solvable groups that are not virtually polycyclic, like the Baumslag--Solitar groups $\operatorname{BS}(1,n)$ for \(|n| > 1\). 

Under the hypothesis of no non-trivial finite normal subgroups, a minimax group $\Gamma$ embeds into a unique $\mathbb Q$-defined linear algebraic group $\HH$, its $\mathbb Q$\emph{-algebraic hull}, first introduced by Arapura and Nori in \cite{AN99}. Deré and Pengitore give the first explicit construction of the hull for this collection of groups in \cite{DP26}. 
 For virtually polycyclic groups, this hull reduces to the notion of $\mathbb Q$\emph{-algebraic hull} introduced by Mostow \cite{Mos70}, and for finitely generated torsion-free nilpotent groups, it may be identified with the \emph{rational Mal'cev completion} \cite{Mal51}. We refer to Subsection~\ref{subsec:minimax} for the precise definitions. 

In this setting, every monomorphism \(\varphi \colon \Gamma \to \Gamma\) uniquely extends to a \(\mathbb Q\)-automorphism of \(\HH\). Applying Theorem \ref{thm:B-LAG} on the level of the \(\mathbb Q\)-algebraic hulls, we then prove our central result:

\begin{letterthm}\label{thm:C-minimax}
	Let $\Gamma=\Gamma_1 \times \Gamma_2 \times  \cdots\times \Gamma_r$ be a direct product of  non-abelian finitely generated virtually solvable minimax groups having \(\mathbb Q\)-algebraic hulls \(\HH_i\) with $\mathbb{Q}$-indecomposable identity component  \(\HH_i^\circ\). Then $\Gamma$ satisfies \CMP{} \emph{on monomorphisms}.
\end{letterthm}

Here, the Central Mixing Property \emph{on monomorphisms} is a stronger version of the one on automorphisms, in the following precise sense:

\medskip 

\textit{Every monomorphism \(\varphi\colon \Gamma \to \Gamma\) uniquely factorizes as
\[
\varphi = \theta \cdot \zeta,
\]
where:   }
\begin{enumerate}[label={(\arabic*)}, ref={(\arabic*)}]
	\item \(\theta\) \textit{is a monomorphism of \(\Gamma\) and there is a unique permutation \(\sigma \in \mathcal S_r\) such that \(\theta\) sends each factor \(\Gamma_i\) into a factor \(\Gamma_{\sigma(i)}\), where \(\Gamma_i\) and \(\Gamma_{\sigma(i)}\) have \(\mathbb Q\)-isomorphic \(\mathbb Q\)-algebraic hulls.}  \label{item:introthmCharacterizationAutomorphismDirectProductOfMinimaxGroups(1)} 
	\item \(\zeta\) \textit{is a central endomorphism of \(\Gamma\) such that \(\pi_{\sigma(i)}\big(\zeta(\Gamma_i)\big) = \{e\}\).}  \label{item:introthmCharacterizationAutomorphismDirectProductOfMinimaxGroups(2)} 
\end{enumerate}
\textit{Conversely, every such pair \((\theta, 
\zeta)\) defines a monomorphism \(\varphi:= \theta \cdot \zeta\) of \(\Gamma\).}

\smallskip  

\noindent \textit{Moreover, if \(\varphi \in \operatorname{Aut}(\Gamma)\), then \(\theta \in \operatorname{Aut}(\Gamma)\) sends each factor \(\Gamma_i\) onto an isomorphic factor \(\Gamma_{\sigma(i)}\).}

\medskip 

Theorem \ref{thm:C-minimax} indeed specializes to \cite[Theorem~3.1]{Sen24}: for a finitely generated torsion-free nilpotent group \(\Gamma\), the algebraic hull \(\HH\) is a connected unipotent \(\mathbb Q\)-group, where the group of rational points \(\HH(\mathbb Q)\) equals the \emph{rational Mal'cev completion} of \(\Gamma\). In this case, the unipotent group \(\HH\) is \(\mathbb Q\)-indecomposable if and only if the rational Mal'cev completion \(\HH(\mathbb Q)\) is directly indecomposable. 
	
\medskip 

 In Subsection~\ref{subsec:assumptionsDECOMP}, we show that the \(\mathbb Q\)-indecomposability assumption on the components \(\HH_i^\circ\) is sharp: both directions of \CMP{} may already fail for directly indecomposable torsion-free nilpotent and polycyclic groups whose hulls are Zariski connected and \(\mathbb Q\)-decomposable. Nevertheless, Subsection~\ref{subsec:example} shows that the hypotheses of Theorem~\ref{thm:C-minimax} still hold for many natural non-nilpotent torsion-free polycyclic groups; we include an abelian-by-cyclic example of the form \(\mathbb Z^4 \rtimes \mathbb Z\) and a Heisenberg-by-cyclic example of the form \(H_5(\mathbb Z) \rtimes \mathbb Z\).

 An interesting direction for further research is to study \CMP{} on automorphisms for direct products of solvable groups beyond the virtually minimax setting, for instance solvable groups that are linear over some field but not over \(\mathbb Q\). The methods of this paper do not directly extend to this setting, since they rely on the existence and functorial properties of the \(\mathbb Q\)-algebraic hull, for which no analogue is currently available in this generality.

\paragraph{Applications.} Finally, as applications of Theorem \ref{thm:C-minimax}, we relate several properties of monomorphisms of a direct product $\Gamma = \Gamma_1  \times \Gamma_2 \times \cdots \times \Gamma_r$ of minimax groups to the corresponding properties of its factors. We give two applications: the co-Hopfian property, and Nielsen--Reidemeister fixed point theory and its relation to twisted conjugacy. 

Recall that a group \(\Gamma\) is called \emph{co-Hopfian} if every monomorphism \(\varphi\colon \Gamma \to \Gamma\) is an automorphism. 

\begin{letterthm} \label{thm:D-cohopfian}
	Let $\Gamma=\Gamma_1 \times \Gamma_2 \times \cdots\times \Gamma_r$ be as in Theorem \ref{thm:C-minimax}. Then \(\Gamma\) is co-Hopfian if and only if every factor \(\Gamma_i\) is co-Hopfian.
\end{letterthm}

A second application of Theorem~\ref{thm:C-minimax} is the study of \emph{twisted conjugacy} and more specifically the Reidemeister number \(R(\varphi) \in \mathbb N_{>0} \cup \{\infty\}\) of a monomorphism \(\varphi\colon \Gamma \to \Gamma\). This number is an important invariant of endomorphisms, related to the study of homotopy classes of fixed points of a continuous map on a topological space and to Nielsen theory as described in \cite{FH94}. 

We note that twisted conjugacy was the original motivation for Senden to study the automorphisms of a direct product of finitely generated torsion-free nilpotent groups in \cite{Sen24}. Also in the setting of Reidemeister numbers, we can completely generalize their results to our class of groups and to all monomorphisms instead of just automorphisms.

\begin{letterthm}
	\label{thm:E-Reidemeister}
	Let $\Gamma=\Gamma_1 \times \Gamma_2 \times  \cdots\times \Gamma_r$ be as in Theorem \ref{thm:C-minimax}. 
	Let $\varphi\colon \Gamma\to\Gamma$ be a monomorphism.  	
	Decompose $\varphi=\theta\cdot\zeta$ as in Theorem~\ref{thm:C-minimax}. Then
	\(
	R(\varphi)=R(\theta).
	\) 
	In particular, if $\sigma=\mathrm{Id}$, then
	\(
	R(\varphi)=\prod_{i=1}^r R(\theta_i),
	\) 
	where $\theta_i\colon \Gamma_i\to\Gamma_i$ are the diagonal monomorphisms induced by $\theta$.
\end{letterthm}

If the permutation of the factors \(\sigma\) is non-trivial, \(R(\varphi) = R(\theta)\) remains easily computable in terms of Reidemeister numbers of monomorphisms of some of the factors,  associated with a disjoint cycle decomposition of \(\sigma\). 

Again, the \(\mathbb Q\)-indecomposability assumption cannot be weakened to mere direct indecomposability: in the case of a \(\mathbb Q\)-decomposable hull, central off-diagonal mixing terms may change the Reidemeister number. This already happens in elementary torsion-free polycyclic examples of the form \(\mathbb Z^4 \rtimes \mathbb Z\).  

\medskip  

Associated with Reidemeister numbers, one studies the \textit{Reidemeister spectrum} of a group \(\Gamma\), defined as  \(
\operatorname{Spec}_{\mathrm{R}}(\Gamma) = \{R(\varphi) \mid \varphi \in \operatorname{Aut}(\Gamma)\}. 
\) 
The group \(\Gamma\) has the \textit{\(R_\infty\)-property} if and only if \(\operatorname{Spec}_{\mathrm{R}}(\Gamma) = \{\infty\}\). Just as in \cite{Sen24}, we may then use Theorem \ref{thm:E-Reidemeister} to describe the Reidemeister spectrum of such a direct product of groups in terms of the spectra of the factors. To do so, we define the \emph{multiplication} of subsets $A_1,\ldots, A_r \subset \mathbb N_{>0} \cup \{\infty\}$ by
 \(
 \prod_{i=1}^r A_i  :=\{a_1 \cdots a_r\mid a_i \in A_i\}
 \) 
 and the \(r^{\text{\textit{th}}}\)~\textit{power} of a subset \(A\subset \mathbb N_{>0} \cup \{\infty\}\) by \(A^{(r)} :=  \prod_{i=1}^r A\).

\begin{lettercor} 
	\label{cor:F-spectrum}
	Let \(\Gamma_1,\Gamma_2, \ldots, \Gamma_k\) be \emph{non-isomorphic} non-abelian finitely generated virtually solvable minimax groups and let \(r_1,r_2, \ldots, r_k \geq 1\).  Assume that each \(\Gamma_i\) has a \(\mathbb Q\)-algebraic hull \(\HH_i\) with \(\mathbb Q\)-indecomposable identity component \(\HH_i^\circ\).  Put \(\Gamma = \Gamma_1^{r_1} \times \Gamma_2^{r_2} \times  \cdots \times \Gamma_k^{r_k}\). Then 
	\[
	\mathrm{Spec}_{\mathrm{R}}(\Gamma)=\prod_{i=1}^k \left( \bigcup_{j = 1}^{r_i}  \mathrm{Spec}_{\mathrm{R}}(\Gamma_i)^{(j)} \right).
	\]
	In particular, \(\Gamma\) has the \(R_\infty\)-property if and only if  some factor \(\Gamma_i\) has the \(R_\infty\)-property. 
\end{lettercor}

\section{Preliminaries}\label{sec:PreliminariesLAG}
\paragraph{Notation.} 
  Throughout, the symmetric group on $r$ letters is denoted by $\mathcal S_r$ and $\sigma \in \mathcal S_r$ is always a general permutation. 

Linear algebraic groups are denoted by bold capital letters such as $\G,\HH$, and morphisms between them by capital Greek letters such as $\Phi, \Theta, \Psi$. 
Their Lie algebras are denoted by fraktur letters such as $\mathfrak g,\mathfrak h$, with induced Lie algebra morphisms denoted by $\Phi_\ast, \Theta_\ast, \Psi_\ast$. 

Finite-dimensional Lie algebras over an arbitrary field are denoted by calligraphic letters such as $\mathcal L,\mathcal M$, with morphisms typically denoted by \(f,g,h\). 

Finally, finitely generated groups are denoted by capital Greek letters such as $\Gamma$, with homomorphisms typically denoted by \(\varphi,\theta,\zeta\).

\paragraph{Linear algebraic \texorpdfstring{\(K\)}{K}-groups in characteristic 0.}

A \textit{linear algebraic group} defined over a field \(K \subset \mathbb C\), also called a \(K\)\textit{-group}, is a subgroup \(\G \subset \operatorname{GL}_n(\mathbb C)\) given by the zeros of a finite number of polynomials over \(K\), equipped with the Zariski topology. Given a subring \(R\) of \(\mathbb C\), we denote the group of \(R\)-points of \(\G\) by \(\G(R) = \G \cap \operatorname{GL}_n(R)\).

We call a map \(\Phi\colon \G\to \HH\) between \(K\)-groups \emph{\(K\)-defined} if it is given by polynomial coordinate functions with coefficients in \(K\), and we call it a \(K\)\emph{-(endo)morphism} if it is an (endo)morphism of \(K\)-groups. In characteristic 0, every bijective \(K\)-morphism is a \(K\)-isomorphism, i.e.\ has a \(K\)-defined inverse (\cite[Proposition 6.13]{Bor91}). 
 We denote by \(\G^\circ\) the Zariski identity component, which is a finite index normal subgroup of \(\G\). If \(\G = \G^\circ\), i.e. \(\G\) is Zariski connected, then we call \(\G\) a \textit{connected \(K\)-group}.   For more information on linear algebraic groups, we refer to the classical textbooks \cite{Bor91,Spr09}.

\paragraph{The Lie algebra of a \texorpdfstring{\(K\)}{K}-group.} To each \(K\)-group \(\G\) we may associate a Lie algebra \(\mathfrak g\). A \(K\)-morphism of \(K\)-groups \(\Phi\colon \G \to \HH\) induces a Lie algebra homomorphism between their respective Lie algebras \(\Phi_\ast \colon \mathfrak g \to \mathfrak h\) by differentiating at \(e_{\G}\). It will be helpful for us to give the following notions a special name:

\begin{definition}[\(K\)-algebraic Lie subalgebras and morphisms]
	Let \(\G\) and \(\HH\) be \(K\)-groups with Lie algebras \(\mathfrak g\) and \(\mathfrak h\), respectively. 
	\begin{itemize}[leftmargin=1.8em]
	\item A Lie algebra morphism \(f\colon \mathfrak g \to \mathfrak h\) is called \(K\)-\emph{algebraic} if there exists a \(K\)-morphism \(\Phi\colon \G \to \HH\) such that \(f = \Phi_\ast\).
	\item  A Lie subalgebra \(\mathfrak l \subset \mathfrak g\) is called \(K\)-\emph{algebraic} if there exists a Zariski closed \(K\)-subgroup \(\mathbf{L} \subset \G\) whose Lie algebra is \(\mathfrak l\). We denote \(\mathfrak l = \operatorname{Lie}(\mathbf{L})\). 
	\end{itemize} 
\end{definition}

We will repeatedly use the following fact, giving a correspondence between the closed connected subgroups of a linear algebraic group \(\G\) and the algebraic Lie subalgebras of its Lie algebra \(\mathfrak g\) (see \cite[\S7.1]{Bor91}).

\begin{theorem}\label{thm:LieAlgebraCorrespondence} Let \(\G\) be a \(K\)-group with Lie algebra \(\mathfrak g\). Let $\HH_1,\HH_2\le \G$ be closed $K$-subgroups.
	\begin{enumerate}
		\item One has 
		\[
		\operatorname{Lie}(\HH_1 \cap \HH_2) = \operatorname{Lie}(\HH_1) \cap \operatorname{Lie}(\HH_2).
		\]
		\item If \(\HH_1\) is connected, then 
		\[
		\HH_1 \leq \HH_2 \iff \operatorname{Lie}(\HH_1) \subset \operatorname{Lie}(\HH_2).
		\]
		\end{enumerate}
		In particular, if both \(\HH_1\) and \(\HH_2\) are connected, then
\[
\HH_1 = \HH_2 \iff \operatorname{Lie}(\HH_1) = \operatorname{Lie}(\HH_2).
\]
\end{theorem}

\paragraph{The unipotent radical and tori.} Let \(\G\) be a \(K\)-group. Each element \(g \in \G\) admits a unique multiplicative \emph{Jordan decomposition} \(g =  g_u g_s\), where \(g_u \in \G\) is unipotent, \(g_s \in \G\) is semisimple and \(g_u, g_s\) commute. The \emph{unipotent radical} \(\mathbf U(\G)\) is the set of unipotent elements of the maximal closed connected normal solvable subgroup (the \emph{radical}) of \(\G\); it is a connected unipotent \(K\)-closed normal subgroup.  
A \(K\)-\emph{torus} \(\mathbf T\) is a connected abelian \(K\)-group consisting entirely of semisimple elements.

Let \(\HH\) be a virtually solvable \(K\)-group. Then \(\HH^\circ\) is a connected solvable \(K\)-group. In this setting, \(\mathbf U(\HH)\) equals the set of unipotent elements of \(\HH\) and \(\HH / \mathbf U(\HH)\) consists entirely of semisimple elements. In fact, there exists a \(K\)-closed subgroup \(\mathbf{S}\), called a \emph{Levi subgroup}, such that \(\HH = \mathbf U(\HH) \rtimes \mathbf{S}\). If \(\HH\) is connected, then \(\HH/\mathbf{U}(\HH)\) is a \(K\)-torus.

\section{\texorpdfstring{\(K\)}{K}-algebraic Lie algebras}\label{sec:CentralMixingPropertyLieAlgebras}

\subsection{A \texorpdfstring{\(K\)}{K}-algebraic Krull-Schmidt theorem}\label{subsec:KAlgebraicKS}

Fix a subfield \(K \subset \mathbb C\). Now we consider the situation of \(K\)-algebraic Lie algebras. Let \(\G\) be a \(K\)-group with Lie algebra \(\mathfrak g\). We call \(\mathfrak g\) \emph{\(K\)-algebraically indecomposable} if \(\mathfrak g\) cannot be written as a direct sum of two non-zero \(K\)-algebraic ideals. In this subsection, we prove a \(K\)-algebraic version of a Krull-Schmidt-type theorem for Lie algebras.

\begin{theorem}[{\(K\)-algebraic Krull-Schmidt theorem}]\label{thm:KalgebraicKrullSchmidt}
	Let \(\G\) be a \(K\)-group with Lie algebra \(\mathfrak g\). Suppose there are two decompositions 
\(\G = \prod_{i=1}^r \G_i = \prod_{j=1}^s \HH_j\) with corresponding Lie algebras	\[
	\mathfrak g = \mathfrak g_1 \oplus \mathfrak g_2 \oplus \ldots \oplus \mathfrak g_r = \mathfrak h_1 \oplus \mathfrak h_2 \oplus \ldots \oplus \mathfrak h_s \]
	into \(K\)-algebraic Lie ideals,  each of which is {\(K\)-algebraically indecomposable}. Let \(\pi_i\) (resp.\ \(\rho_j\)) denote the projection onto \(\mathfrak g_i\) (resp.\ \(\mathfrak h_j\)) along the first (resp.\ second) decomposition. Then \(r = s\), and there exists a permutation \(\sigma \in \mathcal S_r\) such that whenever \(j = \sigma(i)\), we have that:	\begin{enumerate}[label=\emph{(\arabic*)}, ref={(\arabic*)}]
		\item The projection \(\pi_i|_{\mathfrak h_j} \colon \mathfrak h_j \to \mathfrak g_i\) is a \(K\)-algebraic Lie algebra isomorphism, \label{item:thmKSi}
		\item \([\mathfrak g_i, \mathfrak g_i] = [\mathfrak h_j, \mathfrak h_j]\), \label{item:thmKSii}
		\item \(\displaystyle{\mathfrak h_j \subset \mathfrak g_i \oplus Z\Big(\bigoplus_{k \neq i} \mathfrak g_k \Big)}\). \label{item:thmKSiii}
	\end{enumerate}
	If moreover the \(\mathfrak g_i\) are all {non-abelian}, then \(\sigma \in \mathcal S_r\) is unique, in \emph{\ref{item:thmKSi}} the projection \(\rho_j|_{\mathfrak g_i}\colon \mathfrak g_i \to \mathfrak h_j\) is also a \(K\)-algebraic Lie algebra isomorphism, and in \emph{\ref{item:thmKSiii}} we also have \(\displaystyle{\mathfrak g_i \subset \mathfrak h_j \oplus Z\Big(\bigoplus_{k \neq j} \mathfrak h_k \Big)}\).
\end{theorem}
\begin{remark}
	This is a \(K\)-algebraic version of the Krull-Schmidt theorem \cite[Theorem 3.3]{FGH13} over \(\mathbb R\). Using similar, but simpler, arguments, essentially by leaving out all references to linear algebraic groups and \emph{\(K\)-algebraic}, one may prove a version of Theorem~\ref{thm:KalgebraicKrullSchmidt} for abstract Lie algebras over any arbitrary field.
	
	Notably,  \cite[Theorem 3.3]{FGH13} contains an error in case some factors are abelian: it claims in {\ref{item:thmKSi}} that \(\rho_j|_{\mathfrak g_i}\colon \mathfrak g_i \to \mathfrak h_j\) is also an isomorphism without the assumption that the factors are non-abelian. It is however false in general that the two projections \emph{simultaneously} induce isomorphisms; this fails as soon as \(\mathfrak g\) has at least three indecomposable abelian summands, even for general abstract Lie algebras over an arbitrary field, as the next example demonstrates.
\end{remark}
\begin{example}\label{example:counterexampleFHS}
	Let \(F\) be an arbitrary field and consider the abelian Lie algebra \(\mathcal L = F^3\). Let \(e_1,e_2,e_3\) be the standard basis of \(\mathcal L\), i.e.\ the columns of the \(3\times 3\) unit matrix \(I_3\). Let \(a_1,a_2,a_3\) be an alternative basis of \(\mathcal L\), consisting of the columns of the matrix
	\[
	A = \begin{bmatrix}
		1 & 1 & 1 \\
		1 & 1 & 0 \\
		1 & 0 & 1
	\end{bmatrix} \implies A^{-1} = \begin{bmatrix}
		-1 & 1 & 1 \\
		1 & 0 & -1 \\
		1 & -1 & 0
	\end{bmatrix}.
	\]
	Then 
	\[
	\mathcal L = F e_1 \oplus F e_2 \oplus F e_3 = F a_1 \oplus F a_2 \oplus F a_3
	\]
	are two decompositions of \(\mathcal L\) into directly indecomposable ideals. Let \(\pi_i\) denote the projections onto \(Fe_i\) along the first decomposition, and let \(\rho_j\) denote the projections onto \(F a_j\) along the second decomposition. We have the following equivalent statements:
	\begin{itemize}[leftmargin=1.8em]
		\item	\(\pi_i|_{F a_j}\colon Fa_j \to Fe_i\) is an isomorphism \(\iff \pi_i(a_j) \neq 0 \iff A_{ij} \neq 0\), 
		\item	\(\rho_j|_{F e_i} \colon Fe_i \to F a_j \) is an isomorphism \(\iff \rho_j(e_i) \neq 0 \iff (A^{-1})_{ji} \neq 0\).
	\end{itemize}
	However, \cite[Theorem 3.3]{FGH13} for \(F = \mathbb R\) claims that there exists a permutation \(\sigma \in \mathcal S_3\) so that \emph{simultaneously}
	\[
	A_{ij} \neq 0 \quad \text{and} \quad (A^{-1})_{ji} \neq 0	\quad \text{whenever \(j = \sigma(i)\).}
	\]
	This is clearly false: for both \(j=2\) and \(j=3\), the corresponding index \(i\) with \(j = \sigma(i)\) is forced to equal \(1\), contradicting that \(\sigma\) is a permutation.  \demo
\end{example}

For the proof of Theorem~\ref{thm:KalgebraicKrullSchmidt}, it is helpful to use the notion of a \emph{normal} Lie algebra endomorphism.

\begin{definition}[Normal endomorphism]\label{def:normalEndomorphisms}
	Let \(\mathcal L\) be a Lie algebra and \(f\) be a Lie algebra endomorphism, then we call \(f\) a \emph{normal endomorphism} if
	\[
	f([X,Y]) = [f(X),Y]	\quad \text{for all \(X,Y \in \mathcal L\).}
	\]
The projections \(\pi_i\colon \mathfrak g \to \mathfrak g\) onto the ideals \(\mathfrak g_i\) along a direct sum decomposition \(\mathfrak g = \bigoplus_{i=1}^r \mathfrak g_i\) are examples of normal endomorphisms. Notice that the composition of two normal endomorphisms is again a normal endomorphism. If the sum \(f + g\) of two normal endomorphisms is again a Lie algebra endomorphism, then \(f + g\) is again normal, as is the case with sums of distinct projections.  For these normal endomorphisms, we have a version of Fitting's lemma. 
\end{definition}

\begin{lemma}[\(K\)-algebraic Fitting lemma]\label{lemma:KalgebraicFittingLemma}
	Let \(\G\) be a \(K\)-group with Lie algebra \(\mathfrak g\) and \(\Phi_\ast\) be a \(K\)-algebraic normal endomorphism of \(\mathfrak g\). Then there exists an \(N \geq 1\) such that 
	\[
	\mathfrak g = \operatorname{Ker} (\Phi_\ast^N) \oplus \Phi_\ast^N(\mathfrak g),
	\]
	where \(\operatorname{Ker} (\Phi_\ast^N)\) and \(\Phi_\ast^N(\mathfrak g)\) are \(K\)-algebraic Lie ideals.
\end{lemma}
The proof is routine. We include it for completeness.
\begin{proof}
	Using the classical Fitting lemma for vector spaces, we find \(N \geq 1\) such that \begin{equation*}
		\mathfrak g = \operatorname{Ker} (\Phi_\ast^N) \oplus \Phi_\ast^N(\mathfrak g) \qquad \text{as vector spaces.}
		\end{equation*} 
		Using the normality of \(\Phi_\ast^N\), we deduce that \(\Phi_\ast^N(\mathfrak g)\) is a Lie ideal. Hence the direct sum \(\mathfrak g = \operatorname{Ker} (\Phi_\ast^N) \oplus \Phi_\ast^N(\mathfrak g)\) is a direct sum of Lie algebras. 
	If \(\Phi\) is the \(K\)-endomorphism of \(\G\) corresponding to \(\Phi_\ast\), then \(\operatorname{Ker} (\Phi_\ast^N) \) and  \(\Phi_\ast^N(\mathfrak g)\) are the \(K\)-algebraic Lie algebras of the \(K\)-subgroups \(\operatorname{Ker} (\Phi^N) \) and \(\Phi^N(\G)\), respectively. 
\end{proof}

 We obtain the following lemma as an immediate consequence.

\begin{lemma}
	Let \(\G\) be a \(K\)-group with Lie algebra \(\mathfrak g\). If \(\mathfrak g\) is \(K\)-algebraically indecomposable, then any \(K\)-algebraic normal endomorphism \(\Phi_\ast\) of \(\mathfrak g\) is either invertible or nilpotent.
\end{lemma}
The proof of the following lemma contains most of the work to prove Theorem \ref{thm:KalgebraicKrullSchmidt}.
\begin{lemma}\label{lemma:KalgebraicKrullSchmidt}
	Let \(\mathfrak g = \bigoplus_{i=1}^r \mathfrak g_i = \bigoplus_{j=1}^s \mathfrak h_j\) be as in Theorem \ref{thm:KalgebraicKrullSchmidt}. Let \(\pi_i\) and \(\rho_j\) denote the projections onto the subspaces \(\mathfrak g_i\) and \(\mathfrak h_j\) along their respective decompositions. If \(1 \leq i \leq r\), then there exists some \(1 \leq j \leq s\) such that:
	\begin{enumerate}[label=\emph{(\arabic*)}, ref={(\arabic*)}]
		\item The projections \({\pi_i}{\mid_{\mathfrak h_j}}\colon \mathfrak h_j \to \mathfrak g_i\) and
		 \({\rho_j}{\mid_{\mathfrak g_i}}\colon \mathfrak g_i \to \mathfrak h_j\) are  \(K\)-algebraic Lie algebra isomorphisms,  \label{item:LKSi}
		\item \([\mathfrak g_i, \mathfrak g_i] = [\mathfrak h_j, \mathfrak h_j]\), \label{item:LKSii}
		\item \(\displaystyle{\mathfrak h_j \subset \mathfrak g_i \oplus Z\Big(\bigoplus_{k \neq i} \mathfrak g_k \Big)}\) {and} \(\displaystyle{\mathfrak g_i \subset \mathfrak h_j \oplus Z\Big(\bigoplus_{k \neq j} \mathfrak h_k \Big)}\). \label{item:LKSiii}
	\end{enumerate}
		Moreover, if \(\mathfrak g_i\) is \textbf{non-abelian}, then \emph{\ref{item:LKSii}} forces \(j\) to be unique.
\end{lemma}
\begin{proof}
 The projections \(\pi_i\) and \(\rho_j\) induce \(K\)-algebraic Lie algebra morphisms \[
	\varphi_{ij} \colon \mathfrak h_j \longrightarrow \mathfrak g_i, \quad \psi_{ji}\colon \mathfrak g_i \longrightarrow \mathfrak h_j, \quad   \varphi_{ij}\psi_{ji}\colon \mathfrak g_i \longrightarrow \mathfrak g_i, \quad \psi_{ji}\varphi_{ij}\colon \mathfrak h_j \longrightarrow \mathfrak h_j,
	\]
	where \(\varphi_{ij}\psi_{ji}\) and \(\psi_{ji}\varphi_{ij}\) are moreover normal endomorphisms. Note that we have identities
	\begin{equation}
	\mathrm{Id}_{\mathfrak g_i} = \sum_{j=1}^s \varphi_{ij}\psi_{ji} \quad \text{for all \(1 \leq i\leq r\).} \label{eq:identityMatrixofMaps}
	\end{equation}
	Fix \(i\). Since \(\mathfrak g_i\) is \(K\)-algebraically indecomposable, each endomorphism \(\varphi_{ij}\psi_{ji}\) is either invertible or nilpotent. Note that \(\varphi_{ij}\psi_{ji}\) is nilpotent if and only if \(\psi_{ji}\varphi_{ij}\) is, so these maps are either both invertible or both nilpotent.

	The trace of a nilpotent map is \(0\), so taking traces in Equation \eqref{eq:identityMatrixofMaps}, we deduce that there exists at least one \(j\) such that \(\varphi_{ij}\psi_{ji}\) is invertible. Fix such a \(j\). As \(\varphi_{ij}\psi_{ji}\) and \(\psi_{ji}\varphi_{ij}\) are invertible linear maps, so must \(\varphi_{ij}\colon \mathfrak h_j \to \mathfrak g_i\) and \(\psi_{ji} \colon \mathfrak g_i \to \mathfrak h_j\) be. Being restrictions of the respective projection maps, this proves {\ref{item:LKSi}}.
	
	\bigskip 
	
	 \noindent {\ref{item:LKSii}}. Since \(\pi_i\) is a normal endomorphism of \(\mathfrak g\), we have that \[[\mathfrak g_i, \mathfrak g_i] = [\mathfrak g_i, \pi_i(\mathfrak h_j)] = \pi_i\big([\mathfrak g_i, \mathfrak h_j]\big).\] But \([\mathfrak g_i, \mathfrak h_j] \subset \mathfrak g_i\) so \(\pi_i\big([\mathfrak g_i, \mathfrak h_j]\big) = [\mathfrak g_i, \mathfrak h_j] \subset [\mathfrak h_j, \mathfrak h_j].\)  Hence, \([\mathfrak g_i, \mathfrak g_i] \subset [\mathfrak h_j, \mathfrak h_j]\). A symmetric argument shows the reverse inclusion, thus \([\mathfrak g_i, \mathfrak g_i] = [\mathfrak h_j, \mathfrak h_j]\).
	
	\bigskip 
	
	\noindent {\ref{item:LKSiii}}. It is straightforward to show that \([(\mathrm{Id} - \pi_i)(\mathfrak h_j), \mathfrak g] \subset (\mathrm{Id} - \pi_i) [\mathfrak h_j, \mathfrak h_j]\), and since \([\mathfrak g_i, \mathfrak g_i] = [\mathfrak h_j, \mathfrak h_j]\), it follows that  \([(\mathrm{Id} - \pi_i)(\mathfrak h_j), \mathfrak g] = 0.\) Since \((\mathrm{Id} - \pi_i)(\mathfrak h_j) \subset \displaystyle \bigoplus_{k \neq i} \mathfrak g_k\), we deduce that \((\mathrm{Id} - \pi_i)(\mathfrak h_j)  \subset Z\big( \displaystyle\bigoplus_{k \neq i} \mathfrak g_k )\). Therefore \({\mathfrak h_j \subset \mathfrak g_i \oplus Z\Big( \displaystyle \bigoplus_{k \neq i} \mathfrak g_k \Big)}\). The inclusion \({\mathfrak g_i \subset \mathfrak h_j \oplus Z\Big( \displaystyle \bigoplus_{k \neq j} \mathfrak h_k \Big)}\) is similar and uses the projection \(\rho_j\).
	
	\medskip 
	
	If \(\mathfrak g_i\) is non-abelian, then \([\mathfrak g_i, \mathfrak g_i] = [\mathfrak h_j, \mathfrak h_j] \neq 0\) lies in no \(\mathfrak h_k\) but \(\mathfrak h_j\). Therefore this condition indeed forces \(j\) to be unique.
\end{proof}

With the help of this lemma, we can prove Theorem \ref{thm:KalgebraicKrullSchmidt}.
\begin{proof}[Proof of Theorem \ref{thm:KalgebraicKrullSchmidt}]
	Consider \(i = 1\). Lemma \ref{lemma:KalgebraicKrullSchmidt} yields some \(j_1\) such that statements \ref{item:thmKSi}, \ref{item:thmKSii} and \ref{item:thmKSiii} hold for the pair \((i,j) = (1,j_1)\). We intend to continue this process inductively. 
	
	First, we claim that 
	\[
	\mathfrak g = \mathfrak g_1 \oplus \Big( \bigoplus_{j \neq j_1} \mathfrak h_j \Big). \tag{A} \label{eq:DecompositionA}
	\]
	To prove this, notice that 
	\[
	\operatorname{Ker} (\Phi^2) = \operatorname{Ker} (\Phi) = \bigoplus_{j \neq j_1 } \mathfrak h_j \quad \text{and} \quad \Phi^2(\mathfrak g) = \Phi(\mathfrak g) = \mathfrak g_1 
	\]
for the normal endomorphism \(\Phi = \pi_1 \rho_{j_1}  \colon \mathfrak g \to \mathfrak g\) and apply the Fitting Lemma~\ref{lemma:KalgebraicFittingLemma}. 

Consider \(i = 2\). By Lemma \ref{lemma:KalgebraicKrullSchmidt} applied to the decompositions \eqref{eq:DecompositionA} and \(\bigoplus_{i =1}^r \mathfrak g_i\), there exists some summand of decomposition \eqref{eq:DecompositionA} isomorphic to \(\mathfrak g_2\) via the projection \(\pi_2.\) As \(\pi_2(\mathfrak g_1) = 0,\) this summand cannot be \(\mathfrak g_1\), so it must be some \(\mathfrak h_{j_2}\) with \(j_2 \neq j_1\). Lemma  \ref{lemma:KalgebraicKrullSchmidt} further  implies that \([\mathfrak g_2, \mathfrak g_2] = [\mathfrak h_{j_2},\mathfrak h_{j_2}]\) and  \({\mathfrak h_{j_2} \subset \mathfrak g_2 \oplus Z\Big(\bigoplus_{i \neq 2} \mathfrak g_i \Big)}\).

Let \(\tau_{j_2}\) denote the projection onto \(\mathfrak h_{j_2}\) along decomposition \eqref{eq:DecompositionA}. Then \({\tau_{j_2}}_{|_{\mathfrak g_2}}\colon \mathfrak g_2 \to \mathfrak h_{j_2}\) defines an isomorphism, and by similarly applying the Fitting Lemma~\ref{lemma:KalgebraicFittingLemma} to the normal endomorphism \(\pi_2 \tau_{j_2}\) with kernel \(\mathfrak g_1 \oplus \big( \bigoplus_{j \neq j_1, j_2} \mathfrak h_j \big)\) and image \(\mathfrak g_2\), we find a new decomposition
\[
\mathfrak g = \mathfrak g_1 \oplus \mathfrak g_2 \oplus \Big( \bigoplus_{j \neq j_1, j_2} \mathfrak h_j \Big).
\] 

This process can be continued inductively. The desired permutation \(\sigma \in \mathcal S_r\) is given by \(i \mapsto j_i\), and at every step, one verifies the three statements \ref{item:thmKSi}, \ref{item:thmKSii} and \ref{item:thmKSiii} in the theorem. 

\medskip  

If all the \(\mathfrak g_i\) are non-abelian, then the uniqueness of \(j\) in Lemma \ref{lemma:KalgebraicKrullSchmidt} implies we may simply apply Lemma \ref{lemma:KalgebraicKrullSchmidt} a total of \(r\) times to deduce that for every \(1 \leq i \leq r\) a unique \(j_i\) may be found satisfying all the conditions of Lemma \ref{lemma:KalgebraicKrullSchmidt}. Since these conditions are symmetric, we may swap the two decompositions to deduce that the assignment \(i \mapsto j_i\) is bijective, hence \(r = s\) and the stronger version of Theorem \ref{thm:KalgebraicKrullSchmidt} follows.
\end{proof}
\subsection{\texorpdfstring{\(K\)}{K}-algebraic automorphisms of a direct sum of Lie algebras}\label{subsec:KalgebraicAutomorphismsOfADirectSumOfLieAlgebras}

This subsection is devoted to the proof of Theorem~\ref{thm:A-KalgebraicLieAlgebras} from the introduction. 

\begin{lemma}\label{lemma:productIsAgainAMorphism}
Let \(\G\) be a connected \(K\)-group with Lie algebra \(\mathfrak g\) and \(\Theta,\Psi\) be two \(K\)-endomorphisms of \(\G\) such that \(\Psi_\ast\) is a central Lie algebra endomorphism. Then \(\Psi\) is central and \(\Theta \cdot \Psi\) is a \(K\)-endomorphism with differential \((\Theta \cdot \Psi)_\ast = \Theta_\ast + \Psi_\ast\). 
\end{lemma}
\begin{proof}
	The image \(\Psi(\G)\) is a connected closed \(K\)-subgroup of \(\G\), with Lie algebra \(\Psi_\ast(\mathfrak g)\). 
	By assumption \(\Psi_\ast(\mathfrak g) \subset Z(\mathfrak g)\). Since \(Z(\mathfrak g)\) is the Lie algebra of \(Z(\G)\), we deduce from the Correspondence Theorem \ref{thm:LieAlgebraCorrespondence} that \(\Psi(\G) \subset Z(\G)\). It is clear that the map \(\Theta \cdot \Psi\) is defined over \(K\). As \(\Psi\) is central, we deduce \(\Theta \cdot \Psi\) is a \(K\)-morphism with differential \((\Theta \cdot \Psi)_\ast = \Theta_\ast + \Psi_\ast\). 
\end{proof}

\begin{theorem}\label{thm:CharacterizationAlgebraicAutomorphismsDirectSumLieAlgebras}
	Let \( \G = \G_1 \times \G_2 \times \cdots \times \G_r \) be a direct product of connected \(K\)-groups with corresponding Lie algebras
	\[
	\mathfrak g = \mathfrak g_1 \oplus \mathfrak g_2 \oplus \ldots \oplus \mathfrak g_r,
	\]
	where each summand \(\mathfrak g_i\) is \(K\)-algebraically indecomposable. Then:
	\begin{enumerate}
	\item  Every \(K\)-algebraic Lie algebra automorphism \(\Phi_\ast\) of \(\mathfrak g\) can be written as
		\[
		\Phi_\ast = \Theta_\ast + \Psi_\ast,
		\]
		where: 
		\begin{enumerate}[label=\emph{(\arabic*)}, ref={(\arabic*)}]
			\item \(\Theta_\ast\) is a \(K\)-algebraic automorphism of \(\mathfrak g\) mapping each summand \(\mathfrak g_i\) onto an isomorphic summand \(\mathfrak g_j\), and \label{item:thmCharacterizationAlgebraicAutomorphismsDirectSumLieAlgebras(1)}
			\item \(\Psi_\ast\) is a \(K\)-algebraic central endomorphism of \(\mathfrak g\). \label{item:thmCharacterizationAlgebraicAutomorphismsDirectSumLieAlgebras(2)}
		\end{enumerate}
\item Conversely, suppose \(\Theta_\ast\) and \(\Psi_\ast\) are \(K\)-algebraic Lie algebra endomorphisms as in \emph{\ref{item:thmCharacterizationAlgebraicAutomorphismsDirectSumLieAlgebras(1)}} and \emph{\ref{item:thmCharacterizationAlgebraicAutomorphismsDirectSumLieAlgebras(2)}} above. If \(\Theta_\ast + \Psi_\ast\) is an invertible linear map, then it is a \(K\)-algebraic Lie algebra automorphism with corresponding \(K\)-endomorphism \(\Theta \cdot \Psi\). 
		\end{enumerate}
\end{theorem}
Here, a \(K\)\emph{-algebraic Lie algebra automorphism} is a \(K\)-algebraic Lie algebra endomorphism which is also invertible. Admitting a \(K\)-algebraic inverse is guaranteed only under additional assumptions, see Proposition~\ref{prop:K-automorphismOfLieAlgebraIFFK-automorphismOfGroup}.
\begin{remark}
	Theorem~\ref{thm:CharacterizationAlgebraicAutomorphismsDirectSumLieAlgebras} is again a \(K\)-algebraic version of the corresponding result \cite[Theorem 3.4]{FGH13}. Again, by using similar, and simpler, arguments, essentially leaving out all references to linear algebraic groups and \emph{\(K\)-algebraic}, one may prove a version of Theorem~\ref{thm:CharacterizationAlgebraicAutomorphismsDirectSumLieAlgebras} and its addendum Proposition~\ref{prop:addendumtothm:CharacterizationAlgebraicAutomorphismsDirectSumLieAlgebras} for abstract Lie algebras over any arbitrary field. 
\end{remark}

\begin{proof}[Proof of Theorem~\ref{thm:CharacterizationAlgebraicAutomorphismsDirectSumLieAlgebras}] \textit{1.} 
	First assume \(\Phi_\ast\) is a \(K\)-algebraic automorphism of \(\mathfrak g\). This yields two decompositions of \(\mathfrak g\) into direct sums of ideals
	\[
	\mathfrak g = \mathfrak g_1 \oplus \mathfrak g_2 \oplus \ldots \oplus \mathfrak g_r = \Phi_\ast (\mathfrak g_1) \oplus \Phi_\ast (\mathfrak g_2) \oplus \ldots \oplus \Phi_\ast (\mathfrak g_r).
	\] Let \(\pi_i\) denote the projections onto \(\mathfrak g_i\) along the first decomposition. By the \(K\)-algebraic Krull-Schmidt theorem (Theorem \ref{thm:KalgebraicKrullSchmidt}), there exists a permutation \(\sigma \in \mathcal S_r\) such that \[\mathfrak g_i \cong \Phi_\ast (\mathfrak g_{\sigma(i)}), \quad 	\pi_i\big(\Phi_\ast (\mathfrak g_{\sigma(i)})\big) = \mathfrak g_i, \quad \text{and} \quad [\mathfrak g_i, \mathfrak g_i] = \big[\Phi_\ast (\mathfrak g_{\sigma(i)}),\Phi_\ast (\mathfrak g_{\sigma(i)})\big].\]
	Define linear maps \[
	\Theta_\ast = \sum_{i=1}^r \pi_i  \Phi_\ast  \pi_{\sigma(i)}  \qquad \text{and} \qquad \Psi_\ast = \sum_{i=1}^r (\mathrm{Id} - \pi_i)  \Phi_\ast  \pi_{\sigma(i)}.
	\]
	By construction, \(\Phi_\ast = \Theta_\ast + \Psi_\ast\). 
	
 For each \(1 \leq i \leq r\), define the surjective Lie algebra morphism \[(\Theta_i)_\ast = \pi_i  \Phi_\ast  \pi_{\sigma(i)}\colon \mathfrak g \to \mathfrak g_i\]
  and note that \((\Theta_i)_\ast\) is \(K\)-algebraic. 
	It is clear that the corresponding \(K\)-morphisms \(\Theta_i\colon \G \to \G_i\) combine to give a \(K\)-endomorphism \(\Theta\colon \G\to \G\) whose differential equals \(\Theta_\ast = \sum_{i=1}^r (\Theta_i)_\ast\). Hence, \(\Theta_\ast\) is a \(K\)-algebraic Lie algebra morphism. By construction, \(\Theta_\ast\) maps each \(\mathfrak g_i\) onto the isomorphic summand \(\mathfrak g_{\sigma^{-1}(i)}\). In particular, \(\Theta_\ast\) is surjective, and by finite dimensionality, it is a Lie algebra automorphism. This proves \ref{item:thmCharacterizationAlgebraicAutomorphismsDirectSumLieAlgebras(1)}.
	
	Theorem \ref{thm:KalgebraicKrullSchmidt} further guarantees that \(\Phi_\ast \mathfrak g_{\sigma(i)} \subset \mathfrak g_i \oplus Z\Big( \bigoplus_{k \neq i } \mathfrak g_k\Big)\), hence each \(\mathrm{Id} - \pi_i\) maps \(\Phi_\ast \pi_{\sigma(i)}(\mathfrak g)\) into \(Z(\mathfrak g)\). Therefore, every \((\mathrm{Id} - \pi_i)  \Phi_\ast  \pi_{\sigma(i)}\) is a central Lie algebra endomorphism, and hence the same holds for their sum \(\Psi_\ast\). Moreover, each \((\mathrm{Id} - \pi_i)  \Phi_\ast  \pi_{\sigma(i)}\) is \(K\)-algebraic. Applying Lemma~\ref{lemma:productIsAgainAMorphism}, we deduce that the central Lie algebra endomorphism \(\Psi_\ast\) is also \(K\)-algebraic, proving \ref{item:thmCharacterizationAlgebraicAutomorphismsDirectSumLieAlgebras(2)}.

	\bigskip 
	
	\noindent \textit{2.} The map \(\Theta_\ast + \Psi_\ast\) is a Lie algebra endomorphism as \(\Psi_\ast\) is central. Lemma~\ref{lemma:productIsAgainAMorphism} implies that it is \(K\)-algebraic.
\end{proof}
If the summands \(\mathfrak g_i\) are non-abelian, then we can always choose \(\Theta_\ast\) and \(\Psi_\ast\) in a specific nice way. The following addendum suffices to prove Theorem~\ref{thm:A-KalgebraicLieAlgebras} from the introduction. 
\begin{proposition}\label{prop:addendumtothm:CharacterizationAlgebraicAutomorphismsDirectSumLieAlgebras}
	Let \(\G = \prod_{i=1}^r \G_i\) and \(\mathfrak g = \bigoplus_{i=1}^r \mathfrak g_i\) be as in Theorem~\ref{thm:CharacterizationAlgebraicAutomorphismsDirectSumLieAlgebras}, and let \(\pi_i\) denote the projections onto \(\mathfrak g_i\).  Assume that the \(K\)-algebraically indecomposable summands \(\mathfrak g_i\) are non-abelian. Then:  
	\begin{enumerate}
	\item Any \(K\)-algebraic Lie algebra automorphism \(\Phi_\ast\) of \(\mathfrak g\) uniquely decomposes as \(\Phi_\ast = \Theta_\ast+\Psi_\ast\) for some \(K\)-algebraic endomorphisms \(\Theta_\ast, \Psi_\ast\) satisfying \emph{\ref{item:thmCharacterizationAlgebraicAutomorphismsDirectSumLieAlgebras(1)}} and \emph{\ref{item:thmCharacterizationAlgebraicAutomorphismsDirectSumLieAlgebras(2)}} above, and the additional condition that \(\pi_j\big( \Psi_\ast(\mathfrak g_i) \big) = 0\) whenever \(\Theta_\ast(\mathfrak g_i) = \mathfrak g_j.\) \label{item:prop:addendumtothm:CharacterizationAlgebraicAutomorphismsDirectSumLieAlgebras1}
	\item 	Conversely, any pair of  \(K\)-algebraic endomorphisms \(\Theta_\ast, \Psi_\ast\) of \(\mathfrak g\) satisfying the above three conditions defines a \(K\)-algebraic automorphism  \(\Phi_\ast := \Theta_\ast +  \Psi_\ast\) with associated \(K\)-endomorphism \(\Theta\cdot \Psi\). \label{item:prop:addendumtothm:CharacterizationAlgebraicAutomorphismsDirectSumLieAlgebras2}
	\end{enumerate}
\end{proposition}
\noindent   To prove this addendum to Theorem \ref{thm:CharacterizationAlgebraicAutomorphismsDirectSumLieAlgebras}, we require the following observations.

First, we say that a Lie algebra \(\mathfrak g\) has a \textit{non-zero \(K\)-algebraic abelian summand} if we can decompose \(\mathfrak g = \mathfrak a \oplus \mathfrak h\) as a direct sum of \(K\)-algebraic Lie ideals where \(\mathfrak a\) is non-zero abelian. Now it is well-known that a finite-dimensional Lie algebra \(\mathcal L\) (over an arbitrary field) has no non-zero abelian summand if and only if \(Z(\mathcal L) \subset [\mathcal L , \mathcal L]\). To get a similar statement for \(K\)-algebraic Lie algebras in characteristic 0, we need a bit more work. 

\begin{lemma}\label{lemma:gAdmitsNoNonzeroAbelianSummandIFFCenterIsInDerived}
	Let \(\G\) be a \(K\)-group with Lie algebra \(\mathfrak g\), then \(\mathfrak g\) has no non-zero \(K\)-algebraic abelian summand if and only if \(Z(\mathfrak g) \subset [\mathfrak g, \mathfrak g]\).
\end{lemma}
\begin{proof}
Since $\operatorname {Lie}(\HH)=\operatorname{Lie}(\HH^\circ)$ for any $K$-closed subgroup $\HH\le \G$, we may replace
each subgroup appearing in a $K$-algebraic direct sum decomposition of \(\mathfrak g\) by its
identity component. In particular, we may assume $\HH\le \G^\circ$ and hence
replace $\G$ by $\G^\circ$; thus we may assume $\G$ is connected.
 
If \(\mathfrak g = \mathfrak a \oplus \mathfrak h\) for some abelian ideal \(\mathfrak a\), then \(Z(\mathfrak g) = \mathfrak a \oplus Z(\mathfrak h)\) and \([\mathfrak g, \mathfrak g] = [\mathfrak h, \mathfrak h] \subset \mathfrak h\). Therefore \(Z(\mathfrak g) \not\subset [\mathfrak g, \mathfrak g]\), unless \(\mathfrak a = 0\).

 Conversely, suppose that \(Z(\mathfrak g) \not\subset [\mathfrak g, \mathfrak g]\).  Hence \(Z(\G)^\circ \not\subset [\G,\G]\).  In characteristic 0, the connected center \(Z(\G)^\circ\) splits as a direct product \[Z(\G)^\circ = \mathbf T \times \mathbf U\]
 with \(\mathbf T\) a \(K\)-torus and \(\mathbf U\) a unipotent \(K\)-group. Write \(\mathfrak t = \operatorname{Lie}(\mathbf T)\) and \(\mathfrak u = \operatorname{Lie}(\mathbf U)\), so that \(Z(\mathfrak g) = \mathfrak t \oplus \mathfrak u\). We distinguish two cases: either \(\mathbf T\) is trivial and hence \(Z(\G)^\circ\) is a unipotent \(K\)-group, \textit{or} \(\mathbf T\) is non-trivial, and hence \(\mathfrak t \neq 0\).
 
 \medskip 
 
\textit{Case 1:} Assume \(Z(\G)^\circ\) is unipotent. Then \(Z(\G)^\circ\) is a \emph{vector group}, i.e.\ \(Z(\G)^\circ \cong (\mathbb G_{\mathrm a})^d\) for some \(d \geq 0\). Hence, there exists a central 1-parameter \(K\)-subgroup \( \mathbb G_{\mathrm a} \cong \mathbf A\subset Z(\G)^\circ\) whose Lie algebra \(\mathfrak a\) satisfies  \(\mathfrak a \cap [\mathfrak g, \mathfrak g] = 0\). We argue that this is a \(K\)-algebraic abelian summand.
 
 Consider the abelianization \[\pi\colon \G \longrightarrow \G_{\mathrm{ab}} =  \G/{[\G,\G]},\]
 which is a surjective \(K\)-morphism of \(K\)-groups. Of course, \(\G_\mathrm{ab}\) is a connected commutative  \(K\)-group and therefore splits as a direct product \[\G_\mathrm{ab} = \mathbf T_\mathrm{ab} \times \mathbf U_\mathrm{ab}\] with \(\mathbf T_\mathrm{ab}\) a \(K\)-torus and \(\mathbf U_\mathrm{ab}\cong (\mathbb G_{\mathrm a})^m\) the unipotent radical of \(\mathbf G_\mathrm{ab}\). As \(0 \neq \pi(\mathbf A) \subset \mathbf{U}_\mathrm{ab}\) is unipotent, there exists a surjective \(K\)-morphism \(f\colon \mathbf U_\mathrm{ab} \to \mathbb G_{\mathrm a}\) with \(f(\pi(\mathbf A)) = {\mathbb G}_{\mathrm{a}}\).
 
 Let \(\Phi\colon \G \to \mathbb G_{\mathrm a}\) be the surjective \(K\)-morphism given by the composition 
 \[
\Phi \colon \: \G \overset{\pi}{\longrightarrow} \G_{\mathrm{ab}} \overset{\text{proj}}{\longrightarrow} \U_\mathrm{ab}  \overset{f}{\longrightarrow} \mathbb G_{\mathrm a}.
	\]
	Let \(\HH = \operatorname{Ker} (\Phi)\), a \(K\)-closed subgroup with Lie algebra \(\mathfrak h := \operatorname{Ker} (\Phi_\ast)\). Then \(\mathfrak h\) has codimension 1. Moreover, as \(\Phi|_{\mathbf A}\colon \mathbf  A \to \mathbb G_{\mathrm{a}}\) is a \(K\)-isomorphism, we have  \(\mathfrak a \cap \mathfrak h = 0\). Therefore, \(\mathfrak g\) is a direct sum of \(K\)-algebraic ideals 
	\(
	\mathfrak g = \mathfrak a \oplus \mathfrak h 
	\),  with \(
	\mathfrak a\) non-zero abelian.
	
	\medskip 
	
	\textit{Case 2:} Assume \(\mathfrak t \neq 0\). We will prove that \(\mathfrak t\) is a \(K\)-algebraic abelian summand of \(\mathfrak g\).  	
By \cite[Lemma 14.2]{Bor91}, we have that \(\T \cap [\G,\G]\) is a finite group. Hence
\[
\mathfrak t \cap [\mathfrak g,\mathfrak g] = \operatorname{Lie}(\T \cap [\G,\G]) =  0.
\] 

As before, consider the abelianization \(\pi\colon \G \to \G_{\mathrm{ab}} = \T_{\mathrm{ab}} \times \U_{\mathrm{ab}}\). Note that \(\mathbf S:= \pi(\T) \subset \T_{\mathrm{ab}}\) is a \(K\)-subtorus. Hence  
\[
\pi|_{\T} \colon \T \longrightarrow \mathbf S
\]
is an \emph{isogeny}, i.e.\ \(\operatorname{Ker} (\pi|_\T) = \T \cap [\G,\G]\) is finite, 
	 and by e.g.\ \cite[Proof of Symmetry Lemma 2.3]{GG09} there exists a \emph{complementary} \(K\)-subtorus \(\mathbf S' \subset \T_{\mathrm{ab}}\), i.e.\ such that 
	\[
	\T_{\mathrm{ab}} = \mathbf S \cdot \mathbf S' \quad \text{with} \quad \mathbf S \cap \mathbf S' \text{ finite.}
	\] 
	Define the quotient map
	\[
	q\colon \T_{\mathrm{ab}} \longrightarrow \mathbf Q:= \T_{\mathrm{ab}} / \mathbf{S}'.
	\] 
	Then \(q|_{\mathbf S} \colon \mathbf S \to \mathbf Q\) is an isogeny with kernel \(\mathbf S \cap \mathbf S'\).

Finally, let \(\Psi\colon \G \to \mathbf Q\) be the surjective \(K\)-morphism defined by the composition 
\[
\Psi\colon \: \G \overset{\pi}\longrightarrow \G_{\mathrm{ab}} \overset{\mathrm{proj}}\longrightarrow \T_{\mathrm{ab}} \overset{q}\longrightarrow \mathbf Q.
\]
Let \(\HH = \operatorname{Ker} (\Psi )\), a \(K\)-closed subgroup with Lie algebra \(\mathfrak h := \operatorname{Ker} (\Psi_\ast)\). Then \(\Psi|_\T \colon \T \to \mathbf Q\) is a \(K\)-isogeny, so its differential \((\Psi|_\T)_\ast\) is an isomorphism, hence \(\mathfrak t \cap \mathfrak h = 0\) and 
\[
\dim \mathfrak g =  \dim \mathfrak h + \dim \operatorname{Lie}(\mathbf Q) = \dim \mathfrak h + \dim \mathfrak t.
\] 
 Therefore, \(\mathfrak g\) is a direct sum of \(K\)-algebraic ideals \(\mathfrak g = \mathfrak t \oplus \mathfrak h\) with \(\mathfrak t\) non-zero abelian. 
\end{proof}

Finally, we record the following elementary observation; the proof is straightforward.
\begin{lemma}[Component morphisms]\label{lemma:componentMorphisms}
	Let \(\G = \G_1 \times \G_2 \times \cdots \times \G_r\) be a direct product of \(K\)-groups. Let \(\Phi\) be a \(K\)-endomorphism of \(\G\). Define \(K\)-morphisms \[\Phi_{ji}= \pi_j \circ \Phi \circ \iota_i\colon \G_i \longrightarrow \G_j,\]  where \(\iota_i\colon \G_i \hookrightarrow \G\) is the canonical inclusion and \(\pi_j \colon \G \to \G_j\) the projection. Then \(\Phi\) is defined by the \emph{matrix of maps} \(\Phi_{ji}\) as 
	\[
	\Phi(x_1,x_2,\ldots,x_r) = \left( \prod_{i=1}^r \Phi_{1i}(x_i), \prod_{i=1}^r \Phi_{2i}(x_i), \ldots, \prod_{i=1}^r \Phi_{ri}(x_i)\right).
	\]
	Moreover, for a fixed \(j\), the closed subgroups \(\Phi_{ji}(\G_i) \leq \G_j\) commute pairwise.
	
	Conversely, any family of \(K\)-morphisms \(\Phi_{ji}\colon \G_i \to \G_j\) with commuting images for every fixed \(j\) determines a unique \(K\)-endomorphism \(\Phi \colon \G \to \G\) by the same formula.
\end{lemma}
	Analogous results in the categories of abstract groups, Lie groups, and (\(K\)-algebraic) Lie algebras hold verbatim.
\begin{proof}[Proof of Proposition \ref{prop:addendumtothm:CharacterizationAlgebraicAutomorphismsDirectSumLieAlgebras}]
\emph{\ref{item:prop:addendumtothm:CharacterizationAlgebraicAutomorphismsDirectSumLieAlgebras1}}.	We start by proving existence. Let \(\Theta_\ast\) and \(
	\Psi_\ast\) be given as in Theorem~\ref{thm:CharacterizationAlgebraicAutomorphismsDirectSumLieAlgebras}. 	
	By the Lie algebra analog of Lemma \ref{lemma:componentMorphisms}, there exist \(K\)-algebraic Lie algebra isomorphisms \((\Theta_i)_\ast\colon \mathfrak g_i \to \mathfrak g_{\sigma(i)}\) and \(K\)-algebraic central Lie algebra morphisms \((\Psi_{ji})_\ast \colon \mathfrak g_i \to \mathfrak g_j\) which determine the component maps of \(\Theta_\ast\) and \(\Psi_\ast\). 
	
	Split \(\Psi_\ast\) into its \emph{diagonal central part} and \emph{off-diagonal central part} relative to the permutation \(\sigma\):
	\[
	D_\ast := \sum_{i=1}^r \pi_{\sigma(i)} \Psi_\ast \pi_i, \qquad N_\ast := \Psi_\ast - D_\ast. 
	\]
	Then define \[
	\widetilde{\Theta}_\ast := \Theta_\ast + D_\ast, \qquad \widetilde{\Psi}_\ast := N_\ast. 
	\]
	By construction, \[
	\Phi_\ast = \widetilde{\Theta}_\ast + \widetilde{\Psi}_\ast, \qquad  \widetilde{\Theta}_\ast(\mathfrak g_i) \subset \mathfrak g_{\sigma(i)}, \qquad \pi_{\sigma(i)} \big(\widetilde{\Psi}_\ast(\mathfrak g_i)\big)  =0.
	\]
	Moreover, \(\widetilde \Psi_\ast\) remains central. Since \(D_\ast\) is central, Lemma~\ref{lemma:productIsAgainAMorphism} implies that \(	\widetilde{\Theta}_\ast = \Theta_\ast + D_\ast\) is a \(K\)-algebraic Lie algebra endomorphism. 
	By Lemma~\ref{lemma:componentMorphisms}, so is \(\widetilde{\Psi}_\ast\).

It remains to argue that \(\widetilde{\Theta}_\ast\) is still an automorphism.
Define the central endomorphism \[
\Delta_\ast := \Theta_\ast^{-1} D_\ast.
\]
Then \(\widetilde{\Theta}_\ast = \Theta_\ast \big( \mathrm{Id} + \Delta_\ast\big).\) Since each \(\mathfrak g_i\) is non-abelian and \(K\)-algebraically indecomposable, we have \(Z(\mathfrak g_i) \subset [\mathfrak g_i, \mathfrak g_i]\) for each \(i\), hence \(Z(\mathfrak g)\subset [\mathfrak g,\mathfrak g]\). Since any central Lie algebra endomorphism kills \([\mathfrak g,\mathfrak g]\), we have that \(\Delta_\ast^2 = 0\). Hence \(\mathrm{Id} + \Delta_\ast\) is an automorphism with inverse \(\mathrm{Id} - \Delta_\ast\). Since \(\Theta_\ast\) is an automorphism, so is \(\widetilde{\Theta}_\ast = \Theta_\ast \big( \mathrm{Id} + \Delta_\ast\big).\)

\medskip 

Uniqueness is immediate from the component maps: for each \(i\), the condition \(\pi_{\sigma(i)}(\Psi_\ast(\mathfrak g_i))=0\) forces \((\Phi_{\sigma(i)i})_\ast =(\Theta_i)_\ast\), while \((\Phi_{ji})_\ast=(\Psi_{ji})_\ast\) for \(j\neq \sigma(i)\), uniquely determining \(\Theta_\ast\) and \(\Psi_\ast\).

\medskip 

\noindent \emph{\ref{item:prop:addendumtothm:CharacterizationAlgebraicAutomorphismsDirectSumLieAlgebras2}}. As above, we have that \(Z(\mathfrak g) \subset [\mathfrak g, \mathfrak g]\), 
let \(\Delta_\ast := \Theta_\ast^{-1}\Psi_\ast\) and write \(\Phi_\ast = \Theta_\ast (\mathrm{Id} + \Delta_\ast)\) with \(\Delta_\ast^2 =0\). Then \(
(\mathrm{Id} + \Delta_\ast)^{-1} = \mathrm{Id} - \Delta_\ast
\), and \(\Phi_\ast\) is an automorphism with inverse \(
\Phi_\ast^{-1} =(\mathrm{Id} - \Delta_\ast) \Theta_\ast^{-1}.
\) It is \(K\)-algebraic with corresponding \(K\)-morphism \(\Theta \cdot \Psi\) by Lemma~\ref{lemma:productIsAgainAMorphism}. 
\end{proof}

\section{Linear algebraic groups}
As before, fix a subfield \(K \subset \mathbb C\). A \(K\)-group \(\G\) is called \textit{\(K\)-indecomposable} if it cannot be written as a direct product of two non-trivial closed \(K\)-subgroups. Similarly, we say \(\G\) admits a non-trivial abelian \(K\)-factor if \(\G = \mathbf A \times \HH\) for some closed \(K\)-subgroups \(\mathbf A\) and \(\HH\) with \(\mathbf A\) non-trivial abelian.

The goal of this section is to prove the following theorem and proposition using their Lie algebraic analogs Theorem~\ref{thm:CharacterizationAlgebraicAutomorphismsDirectSumLieAlgebras} and Proposition~\ref{prop:addendumtothm:CharacterizationAlgebraicAutomorphismsDirectSumLieAlgebras}, which together suffice to prove Theorem~\ref{thm:B-LAG}.

\begin{theorem} \label{thm:characterizationOfAutomorphismsOfDirectProductOfLinearAlgebraicGroups} Let \(\G = \G_1 \times \G_2 \times \cdots \times \G_r\) 
	be a direct product of connected \(K\)-groups that are \(K\)-indecomposable and have torsion-free centers \(Z(\G_i)(\mathbb{C})\), then every \(K\)-automorphism \(\Phi\) of \(\G\) factorizes as \[\Phi = \Theta \cdot \Psi,\]
	where:
	\begin{enumerate}[label=\emph{(\arabic*)}, ref={(\arabic*)}]
		\item \(\Theta\) is a \(K\)-automorphism of \(\G\) mapping each factor \(\G_i\) onto an isomorphic factor \(\G_j\), and \label{item:thmcharacterizationOfAutomorphismsOfDirectProductOfLinearAlgebraicGroups(1)}
		\item \(\Psi\) is a central \(K\)-endomorphism of \(\G\). \label{item:thmcharacterizationOfAutomorphismsOfDirectProductOfLinearAlgebraicGroups(2)}
	\end{enumerate}
\end{theorem}

\begin{proposition}\label{prop:addendumtothm:characterizationOfAutomorphismsOfDirectProductOfLinearAlgebraicGroups}
	Let \(\G = \prod_{i=1}^r \G_i\) be as in Theorem \ref{thm:characterizationOfAutomorphismsOfDirectProductOfLinearAlgebraicGroups}, and let \(\pi_i\) denote the projections onto \(\G_i\). Assume moreover that the \(K\)-indecomposable factors \(\G_i\) are all non-abelian. Then:
	\begin{enumerate}
		\item	Any \(K\)-automorphism \(\Phi\) of \(\G\) uniquely factorizes as \(\Phi = \Theta \cdot \Psi\) for some \(K\)-endomorphisms \(\Theta, \Psi\) satisfying \emph{\ref{item:thmcharacterizationOfAutomorphismsOfDirectProductOfLinearAlgebraicGroups(1)}} and \emph{\ref{item:thmcharacterizationOfAutomorphismsOfDirectProductOfLinearAlgebraicGroups(2)}} above, and the additional condition that \(\pi_j\big( \Psi(\G_i) \big) = \{e\}\) whenever \(\Theta(\G_i) = \G_j\). 
		\item Conversely, any pair of \(K\)-endomorphisms \(\Theta,\Psi\) of \(\G\) satisfying the above three conditions defines a \(K\)-automorphism \(\Phi:=\Theta\cdot \Psi\). 
	\end{enumerate}
\end{proposition}

The rest of this section is devoted to the proofs of Theorem~\ref{thm:characterizationOfAutomorphismsOfDirectProductOfLinearAlgebraicGroups} and Proposition~\ref{prop:addendumtothm:characterizationOfAutomorphismsOfDirectProductOfLinearAlgebraicGroups}.  
It turns out that if \(\G\) is connected and has torsion-free center \(Z(\G)(\mathbb C)\), then \(\G\) is \(K\)-indecomposable if and only if its Lie algebra \(\mathfrak g\) is \(K\)-algebraically indecomposable, with a similar correspondence holding for admitting a \(K\)-abelian factor or summand. This is summarized in the following result. 
\begin{proposition}\label{prop:K-indecomposableGroupsIFFK-algebraicallyIndecomposableLiealgebras}
	Let \(\G\) be a connected \(K\)-group with torsion-free center and Lie algebra \(\mathfrak g\), then \(\G\) is \(K\)-indecomposable if and only if \(\mathfrak g\) is \(K\)-algebraically indecomposable. 
	Moreover, the following are equivalent:
	\begin{enumerate}[label=\emph{(\arabic*)}, ref={(\arabic*)}]
		\item \(\G\) has no non-trivial abelian \(K\)-factor, \label{item:K-indecomposableGroupsIFFK-algebraicallyIndecomposableLiealgebras1}
		\item \(\mathfrak g\) has no non-zero \(K\)-algebraic abelian  summand, \label{item:K-indecomposableGroupsIFFK-algebraicallyIndecomposableLiealgebras2}
		\item \(Z(\G)^\circ \subset [\G,\G]\), and \label{item:K-indecomposableGroupsIFFK-algebraicallyIndecomposableLiealgebras3}
		\item \(Z(\mathfrak g) \subset [\mathfrak g, \mathfrak g]\). \label{item:K-indecomposableGroupsIFFK-algebraicallyIndecomposableLiealgebras4}
	\end{enumerate}
\end{proposition}

\noindent  Before we prove Proposition~\ref{prop:K-indecomposableGroupsIFFK-algebraicallyIndecomposableLiealgebras}, we  determine which \(K\)-groups have torsion-free center:

\begin{lemma}
	Let \(\G\) be a \(K\)-group, then \(Z(\G)(\mathbb C)\) is torsion-free if and only if \(Z(\G)\) is a unipotent \(K\)-group.
\end{lemma}
\begin{proof}
	Since \(Z(\G)\) is a commutative \(K\)-group, it splits as a direct product \(Z(\G) = \mathbf Z_s \times \mathbf Z_u\), where \(\mathbf Z_s\) is the \(K\)-closed subgroup of semisimple elements and \(\mathbf Z_u = \mathbf U\big(Z(\G) \big)\) is the \(K\)-closed subgroup of unipotent elements (see \cite[Theorem 4.7]{Bor91}). Note that \(\mathbf Z_s\) is diagonalizable by \cite[Proposition 8.4]{Bor91}. If \(\mathbf Z_s\) is non-trivial, then \(\mathbf Z_s(\mathbb C)\) contains non-trivial torsion by \cite[\S8.9]{Bor91}. Thus, if \(Z(\G)(\mathbb C)\) is torsion-free, we must have \(\mathbf Z_s = 1\), so \(Z(\G)\) is unipotent. 	Conversely, if \(Z(\G) = \mathbf Z_u\) is unipotent, then \(Z(\G)(\mathbb C)\) is clearly torsion-free.
\end{proof}
 To prove Proposition~\ref{prop:K-indecomposableGroupsIFFK-algebraicallyIndecomposableLiealgebras}, we require two auxiliary results, Lemma~\ref{lemma:Ghasnonontrivialfinitenormalsubgroups} and Proposition~\ref{prop:K-automorphismOfLieAlgebraIFFK-automorphismOfGroup} below.

\begin{lemma}\label{lemma:Ghasnonontrivialfinitenormalsubgroups}
	Let \(\G\) be a connected \(K\)-group and let \(E\) be a finite normal subgroup, then \(E \leq Z(\G)\). 
	In particular, \(\G\) has torsion-free center if and only if \(\G\) has no non-trivial finite normal subgroups. 
\end{lemma}
\begin{proof} Let \(x\in E\). The orbit map under conjugation \(c_x\colon \G\to E\) defined by \(c_x(g) = gxg^{-1}\)  is a morphism of varieties. As \(\G\) is connected, \(c_x(\G)\) is connected. But \(E\) is finite, so \(c_x(\G)=\{x\}\). Therefore \(x\in Z(\G)\). \end{proof}
\begin{example}\label{example:Ghasnonontrivialfinitenormalsubgroups}
	Connectedness matters: let $\G$ be a non-abelian finite group, viewed (via a faithful
	representation) as a finite, hence disconnected, 0-dimensional $\mathbb Q$-group.
	Then $\G$ is a non-central finite normal subgroup of itself. \demo 
\end{example}
The differential of a \(K\)-automorphism of a \(K\)-group \(\G\) is an automorphism of \(\mathfrak g\). The converse holds only under additional assumptions, e.g.\  connectedness and torsion-free center.
\begin{proposition}\label{prop:K-automorphismOfLieAlgebraIFFK-automorphismOfGroup}
 Let \(\G\) be a connected \(K\)-group with torsion-free center and Lie algebra \(\mathfrak g\). Let \(\Phi\colon \G \to \HH\) be a \(K\)-morphism between connected \(K\)-groups, then \(\Phi\) is a \(K\)-isomorphism if and only if \(\Phi_\ast\colon \mathfrak g \to \mathfrak h\) is a \(K\)-algebraic isomorphism.
\end{proposition}
\begin{proof}
	Assume \(\Phi_\ast\) is an isomorphism. 
	Since \(\Phi_\ast(\mathfrak g) = \mathfrak h\) and \(\Phi(\G)\) is a closed connected subgroup of \(\HH\) with Lie algebra \(\Phi_\ast(\mathfrak g)\), we have \(\Phi(\G) = \HH\). From the dimension theorem 
	\[
	\dim \G = \dim \operatorname{Ker} \Phi + \dim \Phi(\G),
	\]
	it now follows that \(\operatorname{Ker} \Phi\) is zero-dimensional and hence finite. Since \(\G\) has torsion-free center, Lemma \ref{lemma:Ghasnonontrivialfinitenormalsubgroups} implies that \(\operatorname{Ker} \Phi = \{e_{\G}\}\). We conclude that \(\Phi\) is a bijective \(K\)-morphism, and hence a \(K\)-isomorphism.
\end{proof}
\begin{example} Both hypotheses, connectedness and torsion-free center, are essential. 
	\begin{itemize}[leftmargin=1.8em]
		\item \emph{Disconnected with torsion-free center.}
		Let \(\G\) be any finite centerless group (e.g.\ \(\mathcal S_3\)), viewed via a faithful representation as a \(0\)-dimensional, hence disconnected, \(\mathbb Q\)-group.  Then \(Z(\G) = \{e_{\G}\}\) is torsion-free, and any non-bijective endomorphism has bijective differential since  \(\mathfrak g= 0\).

		\item \emph{Connected with torsion in the center.} Take the 1-dimensional \(\mathbb Q\)-torus 
		\(\G = \mathbb G_{\mathrm{m}} \) and fix \(n \geq 2\). The \(\mathbb Q\)-endomorphism \(\Phi\colon t \mapsto t^n\) is not injective, but its differential  \(\Phi_\ast \colon \mathfrak g \to \mathfrak g\) is multiplication by \(n\) and hence bijective in characteristic 0. 		\demo 
	\end{itemize}
\end{example}

\begin{proof}[{Proof of Proposition \ref{prop:K-indecomposableGroupsIFFK-algebraicallyIndecomposableLiealgebras}}] 
	It is clear that if \(\G = \HH \times \mathbf L\) for some closed \(K\)-subgroups \(\HH\) and \(\mathbf L\), then \(\mathfrak g = \mathfrak h \oplus \mathfrak l\) for some \(K\)-algebraic ideals \(\mathfrak h\) and \(\mathfrak l\), namely the Lie algebras of the normal subgroups \(\HH, \mathbf L \leq \G\). If moreover \(\HH\) is non-trivial abelian, then so is \(\mathfrak h\), proving  \ref{item:K-indecomposableGroupsIFFK-algebraicallyIndecomposableLiealgebras2}~\(\Rightarrow\)~\ref{item:K-indecomposableGroupsIFFK-algebraicallyIndecomposableLiealgebras1}.
	
	For the converse implication, suppose \(\mathfrak g = \mathfrak h \oplus \mathfrak l\) for some \(K\)-algebraic Lie ideals \(\mathfrak h\) and \(\mathfrak l\) with corresponding normal connected \(K\)-subgroups \(\HH\) and \(\mathbf L\) of \(\G\). By \cite[\S2.3]{Bor91}, the commutator subgroup \([\HH,\mathbf L]\) is a closed connected normal \(K\)-subgroup of \(\G\) with Lie algebra \([\mathfrak h,\mathfrak l] = 0\) and thus \([\HH,\mathbf L]\) is trivial. Consider the multiplication \(K\)-morphism
	\[
	m\colon \HH \times \mathbf L \longrightarrow \G\colon (h,l) \longmapsto h \cdot l.
	\]
	Since \(m_\ast\) is the identity on \(\mathfrak h \oplus \mathfrak l = \mathfrak g\), a dimension argument shows that \(m\) is surjective, so \(\G = \HH \cdot \mathbf{L}\). Using \([\HH,\mathbf L] = \mathbf 1\), this implies \(Z(\HH) \subset Z(\G)\) and \(Z(\mathbf{L})\subset Z(\G)\), hence \(\HH \times \mathbf{L}\) has torsion-free center. Proposition~\ref{prop:K-automorphismOfLieAlgebraIFFK-automorphismOfGroup} then implies that \(m\) is a \(K\)-isomorphism. If moreover \(\mathfrak h\) is non-trivial abelian, then so must \(\HH\) be, proving  \ref{item:K-indecomposableGroupsIFFK-algebraicallyIndecomposableLiealgebras1}~\(\Rightarrow\)~\ref{item:K-indecomposableGroupsIFFK-algebraicallyIndecomposableLiealgebras2}.

	The equivalence \ref{item:K-indecomposableGroupsIFFK-algebraicallyIndecomposableLiealgebras2}~\(\Leftrightarrow\)~\ref{item:K-indecomposableGroupsIFFK-algebraicallyIndecomposableLiealgebras4} is precisely Lemma \ref{lemma:gAdmitsNoNonzeroAbelianSummandIFFCenterIsInDerived}. The remaining equivalence \ref{item:K-indecomposableGroupsIFFK-algebraicallyIndecomposableLiealgebras3}~\(\Leftrightarrow\)~\ref{item:K-indecomposableGroupsIFFK-algebraicallyIndecomposableLiealgebras4} follows from \(\operatorname{Lie}\big([\G,\G]\big) = [\mathfrak g,\mathfrak g]\), \(\operatorname{Lie}\big(Z(\G)^\circ\big) = Z(\mathfrak g)\), and the Correspondence Theorem \ref{thm:LieAlgebraCorrespondence}.
 \end{proof}

\begin{proof}[Proof of Theorem~\ref{thm:characterizationOfAutomorphismsOfDirectProductOfLinearAlgebraicGroups}]
	The differential \(\Phi_\ast\) is a \(K\)-algebraic automorphism of the corresponding direct sum of Lie algebras \(
	\mathfrak g = \mathfrak g_1 \oplus \mathfrak g_2 \oplus \ldots \oplus \mathfrak g_r,
	\) whose summands are \(K\)-algebraically indecomposable by Proposition \ref{prop:K-indecomposableGroupsIFFK-algebraicallyIndecomposableLiealgebras}. By Theorem \ref{thm:CharacterizationAlgebraicAutomorphismsDirectSumLieAlgebras}, we have that \[\Phi_\ast = \Theta_\ast + \Psi_\ast,\] where \(\Theta_\ast\) is a \(K\)-algebraic Lie algebra automorphism mapping each summand \(
	\mathfrak g_i\) onto \(\mathfrak g_j\) and \(
	\Psi_\ast \) is a \(K\)-algebraic central Lie algebra endomorphism. 
 Proposition~\ref{prop:K-automorphismOfLieAlgebraIFFK-automorphismOfGroup} implies that \(\Theta\) is a \(K\)-automorphism of \(\G\). 	It follows from Lemma~\ref{lemma:productIsAgainAMorphism} that \(\Psi\) is a central \(K\)-endomorphism and that \(\Theta \cdot \Psi \) is a \(K\)-morphism with \[(\Theta \cdot \Psi)_\ast = \Theta_\ast + \Psi_\ast.\]
 
 As \(\Theta_\ast+\Psi_\ast\) is an automorphism of \(\mathfrak g\), Proposition~\ref{prop:K-automorphismOfLieAlgebraIFFK-automorphismOfGroup} implies that \(\Theta \cdot \Psi\) is a \(K\)-automorphism of \(\G\). Hence, \(\Phi\) and \(\Theta \cdot \Psi\) are two \(K\)-automorphisms of a connected linear algebraic group whose differentials coincide on the Lie algebra, and thus \(\Phi = \Theta \cdot \Psi\).

	Finally, let \(i,j\) be given such that \(\Theta_\ast (\mathfrak g_i) = \mathfrak g_j\). Then \(\Theta(\G_i)\) and \(\G_j\) are  closed connected subgroups of \(\G\) with the same Lie algebra. By the Correspondence Theorem \ref{thm:LieAlgebraCorrespondence}, we deduce  \(\Theta(\G_i) = \G_j\). This concludes the proof.
\end{proof}

\begin{proof}[Proof of Proposition~\ref{prop:addendumtothm:characterizationOfAutomorphismsOfDirectProductOfLinearAlgebraicGroups}] 
	As in the proof of Theorem \ref{thm:characterizationOfAutomorphismsOfDirectProductOfLinearAlgebraicGroups}, this follows similarly from the Lie algebra analog, Proposition \ref{prop:addendumtothm:CharacterizationAlgebraicAutomorphismsDirectSumLieAlgebras}.
\end{proof}

In this final part of the section, we prove a variant of Proposition~\ref{prop:addendumtothm:characterizationOfAutomorphismsOfDirectProductOfLinearAlgebraicGroups} for a non-connected virtually solvable case satisfying an additional property. A virtually solvable \(K\)-group \(\HH\) has a \emph{strong unipotent radical} if no non-trivial semisimple element centralizes the unipotent radical \(\U\), so
\[
C_{\HH}(\U) \subset \U. 
\] 
In particular, the center of $\HH$ is unipotent and hence torsion-free. This will be useful in proving Theorem~\ref{thm:C-minimax} in the next section, as the \(\mathbb Q\)-algebraic hull of a finitely generated virtually solvable minimax group has a strong unipotent radical.

\begin{theorem}\label{thm:virtuallysolvablegroups}
	Let \(\HH = \HH_1\times \cdots \times \HH_r\) be a direct product of virtually solvable \(K\)-groups \(\HH_i\) with strong unipotent radicals \(\U_i\). Assume that each identity component \(\HH_i^\circ\) is non-abelian and \(K\)-indecomposable.
	
	\medskip 
	
	\noindent Then every \(K\)-automorphism \(\Phi\colon \HH \to \HH\) uniquely factorizes as 
	\[
	\Phi = \Theta \cdot \Psi,
	\] 
	where: 
		\begin{enumerate}[label=\emph{(\arabic*)}, ref={(\arabic*)}]
		\item \(\Theta\) is a \(K\)-automorphism of \(\HH\) mapping each factor \(\HH_i\) onto an isomorphic factor \(\HH_{\sigma(i)}\) for some unique permutation \(\sigma \in \mathcal S_r\), and \label{item:thm:virtuallysolvablegroups(1)}
		\item \(\Psi\) is a central \(K\)-endomorphism of \(\HH\) satisfying \(\pi_{\sigma(i)}\big(\Psi(\HH_i)\big) = \{e\}\).  \label{item:thm:virtuallysolvablegroups(2)}
	\end{enumerate}
	Conversely, every such pair \((\Theta,\Psi)\) defines a \(K\)-automorphism \(\Phi := \Theta \cdot \Psi\) of \(\HH\). 
\end{theorem}
\begin{proof}
Let \(\Phi\) be a \(K\)-automorphism of \(\HH\). Since each \(\HH_i\) has a strong unipotent radical \(\U_i = \U(\HH_i) = \U(\HH_i^\circ)\), so does each \(\HH_i^\circ\). Hence, every \(\HH_i^\circ\) has a torsion-free center \(Z(\HH_i^\circ)(\mathbb C)\).  

Therefore Proposition~\ref{prop:addendumtothm:characterizationOfAutomorphismsOfDirectProductOfLinearAlgebraicGroups} applies to the restriction \(\Phi^\circ\colon \HH^\circ \to \HH^\circ\). Using the matrix of component maps (Lemma~\ref{lemma:componentMorphisms}): there exists a permutation \(\sigma \in \mathcal S_r\) such that, for a fixed \(j = \sigma(i)\) and \(k \neq i\), \(\Phi_{ji}|_{\HH_i^\circ}\colon \HH_i^\circ \to \HH_{\sigma(i)}^\circ\) is a \(K\)-isomorphism and \(\Phi_{jk}(\HH_k^\circ) \subset Z(\HH_j^\circ)\). 
 Moreover, \(\Phi_{ji}(\HH_i^\circ) = \HH_j^\circ\) and \(\Phi_{jk}(\HH_k)\) are two commuting subgroups of \(\HH_j\). Therefore by the strong unipotent radical property:
 \[
 \Phi_{jk}(\HH_k) \subset C_{\HH_j}(\HH_j^\circ) \subset C_{\HH_j}(\U_j) \subset \U_j \subset \HH_j^\circ.
 \] 
 
 Consider the finite component groups \(\HH_i / \HH_i^\circ\) and \(\HH / \HH^\circ = \HH_1 / \HH_1^\circ \times \cdots \times \HH_r / \HH_r^\circ\).  
Since every off-diagonal component \(\Phi_{jk}(\HH_k)\) of \(\Phi\) lands in \(\HH_j^\circ\), the induced \(\overline \Phi \in  \operatorname{Aut}(\HH / \HH^\circ)\) is a pure permutation of isomorphisms \(\overline{\Phi}_{\sigma(i)i}\) between the factors. Therefore, since \(\Phi_{\sigma(i)i}\colon \HH_i \to \HH_{\sigma(i)}\) induces a \(K\)-isomorphism \(\HH_i^\circ \to \HH_{\sigma(i)}^\circ\) and an isomorphism \(\overline{\Phi}_{\sigma(i)i} \colon \HH_i / \HH_i^\circ \to \HH_{\sigma(i)} / \HH_{\sigma(i)}^\circ\), it follows that each \(\Phi_{\sigma(i)i}\) is a \(K\)-isomorphism mapping \(\HH_i\) onto the isomorphic factor \(\HH_{\sigma(i)}\). Moreover, each off-diagonal \(\Phi_{jk}(\HH_{k})\) commutes with the main image \(\Phi_{ji}(\HH_{i}) = \HH_j\), hence \(\Phi_{jk}(\HH_{k}) \subset Z(\HH_j)\). It follows that \(\Phi\) is of the form \(\Phi = \Theta \cdot \Psi\) as in the statement of the theorem. Uniqueness of the decomposition follows from the matrix of component maps. 

For the converse statement, assume (\(\Theta,\Psi\)) are as in \ref{item:thm:virtuallysolvablegroups(1)} and \ref{item:thm:virtuallysolvablegroups(2)} and put \(\Phi := \Theta \cdot \Psi\). Then 
\[
\Theta(\HH^\circ) = \HH^\circ, \qquad \Psi(\HH)\subset Z(\HH) \subset \U \subset \HH^\circ. 
\]
Hence \(\overline{\Phi} = \overline{\Theta}\) is an automorphism of the finite component group \(\HH/\HH^\circ\). Moreover, \(\Phi|_{\HH^\circ} = \Theta|_{\HH^\circ} \cdot \Psi|_{\HH^\circ}\) is a \(K\)-automorphism of \(\HH^\circ\) by Proposition~\ref{prop:addendumtothm:characterizationOfAutomorphismsOfDirectProductOfLinearAlgebraicGroups}. Therefore \(\Phi\) is a \(K\)-automorphism of \(\HH\).
\end{proof}

\section{Finitely generated virtually solvable minimax groups}\label{sec:minimax}

\subsection{Minimax groups and their \texorpdfstring{\(\mathbb Q\)}{Q}-algebraic hulls}\label{subsec:minimax}

From now on, we will focus on a class of finitely generated virtually solvable groups, known as the \emph{finitely generated virtually solvable minimax groups}. This class of groups strictly includes all virtually polycyclic groups, and hence all finitely generated virtually nilpotent groups. We briefly recall the definition of \emph{minimax.}
\begin{definition}
	Let \(G\) be a group.
	\begin{itemize}[leftmargin=1.8em]
		\item We say \(G\) satisfies \textbf{Max}, also known as the \emph{maximal condition} or \emph{Noetherian} property, if every ascending chain of subgroups stabilizes.  
				\item We say \(G\) satisfies \textbf{Min}, also known as the \emph{minimal condition} or \emph{Artinian} property, if every descending chain of subgroups stabilizes. 
		\item We say \(G\) is \emph{minimax} if it admits a normal series
		\[
		1 = G_0 \lhd G_1 \lhd G_2  \lhd \cdots \lhd G_{k-1} \lhd G_k = G
		\] 	
		such that each factor \(G_i / G_{i-1}\) satisfies \textbf{Max} or \textbf{Min}.
				\end{itemize}
\end{definition}

\begin{remark}	
	The class of finitely generated virtually solvable minimax groups coincides with the class of \emph{finitely generated virtually solvable FAR-groups} in \cite{DP26}, see \cite[Lemma 2.3]{DP26}. 
\end{remark}

\paragraph{\texorpdfstring{\(\mathbb Q\)}{Q}-algebraic hulls.} Let \(\Gamma\) be a finitely generated virtually solvable minimax group. Its \emph{Hirsch length} \(h(\Gamma)\) is defined as in \cite[\S~2.1]{DP26}; it agrees with the usual Hirsch length when \(\Gamma\) is virtually polycyclic.

\begin{definition}[\(\mathbb Q\)-algebraic hull of \(\Gamma\)] \label{def:Qalgebraichull} A \(\mathbb Q\)-algebraic hull of \(\Gamma\) is a pair \((\HH, \iota)\) consisting of a \(\mathbb Q\)-group \(\HH\) and an injective group homomorphism \(\iota \colon \Gamma \hookrightarrow \HH\) satisfying the following three conditions: \begin{enumerate}
		\item  \(\iota(\Gamma)\leq \HH(\mathbb Q)\) and \(\iota(\Gamma)\) is Zariski dense in \(\HH\).
		\item  \(\dim \mathbf U\big( \HH \big) = h(\Gamma)\).
		\item 	\(\HH\) has a \emph{strong unipotent radical}, i.e.\ \(C_{\HH}\big( \U(\HH) \big)\subset \U(\HH)\). 
	\end{enumerate}
\end{definition}
\begin{remark}
	In \cite{DP26}, one additionally requires \(\iota(\Gamma)\leq \HH(\mathbb{Z}[1/S])\), where \(S\) is the \textit{spectrum} of \(\Gamma\). We will only need \(\iota(\Gamma)\leq \HH(\mathbb Q)\), so we do not introduce the spectrum here. 
\end{remark}

 Assume \(\Gamma\) has a \(\mathbb Q\)-algebraic hull \((\HH,\iota)\). We will always identify \(\Gamma\) with its image under \(\iota \colon \Gamma \hookrightarrow \HH\) and simply call \(\HH\) a \(\mathbb Q\)-algebraic hull of \(\Gamma\). By \cite[Corollary 3.8]{DP26},  a \(\mathbb Q\)-algebraic hull \(\HH\) is unique up to a \(\mathbb Q\)-isomorphism preserving \(\Gamma\), so we call \(\HH\) \emph{the} \(\mathbb Q\)-algebraic hull of \(\Gamma\), when it exists. The following two results can be found in \cite{DP26}.

\begin{lemma}\label{lemma:QalgebraicHullOfAFiniteIndexSubgroup} 	Let \(\Gamma' \leq \Gamma\) be a subgroup of finite index, then the Zariski closure of \(\Gamma'\) inside \(\HH\) is a \(\mathbb Q\)-algebraic hull of \(\Gamma'\). 
\end{lemma}

\begin{lemma}
	Let \(\varphi\colon \Gamma \to \Gamma\) be a monomorphism, then \(\varphi\) uniquely extends to a \(\mathbb Q\)-automorphism \(\Phi\colon \HH \to \HH\).
\end{lemma}

See Example~\ref{example:Q-indecomposabilityPolycyclic} for an example of the construction of a \(\mathbb Q\)-algebraic hull, illustrating the procedure in a concrete torsion-free polycyclic example.

\begin{example}\label{example:unipotenthull}
	Let \(\Gamma\) be a finitely generated torsion-free nilpotent group. Its \(\mathbb Q\)-algebraic hull \(\HH\) is a unipotent \(\mathbb Q\)-group, and \(\Gamma^{\mathbb Q}:=\HH(\mathbb Q)\) identifies with the classical \emph{rational Mal'cev completion} \cite{Mal51}. Henceforth, we shall call such a hull \(\HH\) a \emph{unipotent hull}. 
	
Equivalently (see e.g.\ \cite[Chapter 6]{Seg83}, \cite[\S2]{Dek18}, \cite{BMO16}), \(\Gamma^{\mathbb Q}\) is, up to unique isomorphism over \(\Gamma\), the unique torsion-free nilpotent group containing \(\Gamma\) such that
	\begin{itemize}[leftmargin=1.2em]
		\item 	\(\Gamma^\mathbb Q\) is \emph{uniquely divisible}: for every \(g \in \Gamma^\mathbb{Q}\) and \(k \in \mathbb Z_{>0}\) there exists a unique \(r \in \Gamma^\mathbb Q\) with \(r^k = g\); 
		\item  \(\Gamma^\mathbb Q\) is the \emph{radicable closure} of \(\Gamma\): for every \(g \in \Gamma^\mathbb Q\) there exists \(k \in \mathbb Z_{>0}\) such that \(g^k \in \Gamma\).  \demo
	\end{itemize}
\end{example}

\paragraph{Existence of a \texorpdfstring{\(\mathbb Q\)}{Q}-algebraic hull.} Any group \(G\) has a largest torsion normal subgroup containing all other torsion normal subgroups, which we denote by \(\tau(G)\). If \(\tau(G) = 1\), then we call \(G\) \emph{\textsf{WTN}}; that is, \(G\) has no non-trivial torsion normal subgroups.

 A finitely generated virtually solvable minimax group \(\Gamma\) admits a \(\mathbb Q\)-algebraic hull if and only if \(\Gamma\) is \textsf{WTN} by \cite[Theorem 3.6]{DP26}.  Important examples of such groups include torsion-free polycyclic groups and finitely generated torsion-free nilpotent groups.

 Although one usually considers the property of \textsf{WTN}, it turns out that being \textsf{WTN} is equivalent to having no non-trivial \textit{finite} normal subgroups (\textsf{WFN}).

\begin{proposition}\label{prop:WTNiffWFN} 
	Let \(\Gamma\) be a finitely generated virtually solvable minimax group, then \(\Gamma\) is \emph{\textsf{WTN}} if and only if \(\Gamma\) is \emph{\textsf{WFN}}.
\end{proposition}
\begin{proof}
	The implication \textsf{WTN} \(\Rightarrow\) \textsf{WFN} is trivial. For the converse implication, assume that \(\Gamma\) has no non-trivial finite normal subgroups. We show that the maximal normal torsion subgroup \(T:= \tau(\Gamma) \) is trivial.
	
	Note that \(T\) is a \emph{Chernikov group} by \cite[5.2.1]{LR04}. Hence its \emph{finite residual} \[\operatorname{Fr}(T) := \bigcap \, \big\{N \, \big|\,  N \lhd T, [T:N] < \infty\big\}  \] is a characteristic finite index subgroup of  \(T\) by \cite[1.4.2]{LR04}. It thus suffices to argue that \(\operatorname{Fr}(T) = 1\).
	
	Aiming for a contradiction, suppose that \(\operatorname{Fr}(T) \neq 1\). By \cite[Proof of 1.4.1]{LR04}, \(\operatorname{Fr}(T)\) is a direct product of finitely many \emph{Pr\"ufer \(p\)-groups} 
	\[
	\operatorname{Fr}(T) = (C_{{p}_1^\infty}) \times (C_{{p}_2^\infty}) \times \cdots \times (C_{{p}_k^\infty}),
	\]
for some not necessarily distinct prime numbers \(p_1,p_2,\ldots,p_k\). In every factor \(C_{{p}_i^\infty}\), we may consider the unique subgroup \(C_{p_i} \cong \mathbb Z / p_i \mathbb Z\) isomorphic to the cyclic group of order \(p_i\), and form \[N = C_{p_1} \times C_{p_2} \times \cdots \times C_{p_k}.\] 
Then \(N\) is non-trivial and finite, but also characteristic in \(\operatorname{Fr}(T)\), as it equals the subgroup generated by the set of prime order elements of \(\operatorname{Fr}(T)\). Hence \(N\) is a non-trivial finite normal subgroup of \(\Gamma\), a contradiction.
	\end{proof}

\paragraph{The Fitting subgroup.} For any group \(G\), the \emph{Fitting subgroup} \(\operatorname{Fitt}(G)\) is defined by \[\operatorname{Fitt}(G) = \langle N \mid N \lhd G, \, N \text{ nilpotent} \rangle.\] \(\operatorname{Fitt}(G)\) is  \emph{locally nilpotent}, i.e.\ every finitely generated subgroup of \(\operatorname{Fitt}(G)\) is nilpotent. Therefore, if \(G\) is finitely generated, then \(G\) is nilpotent if and only if \(\operatorname{Fitt}(G) = G\).

  Let \(\Gamma\) be a finitely generated virtually solvable minimax group. We will usually denote its Fitting subgroup by \(F = \operatorname{Fitt}(\Gamma)\). In this setting, \(F\) itself is nilpotent and \(\Gamma/F\) is virtually abelian by \cite[5.2.2 \& 4.4.1]{LR04}. 
Assume moreover that \(\Gamma\) admits a \(\mathbb Q\)-algebraic hull \(\HH\). By \cite[Proposition 3.13]{DP26}, it holds that \(F = \Gamma \cap \U\big( \HH \big)\). In this case, it is clear that \(F\) is \textit{injectively characteristic}, in the sense that \(\varphi(F) \leq F\) for every monomorphism \(\varphi\colon \Gamma \to \Gamma\).

\subsection{Proof of Theorem \texorpdfstring{\ref{thm:C-minimax}}{C}}

\begin{proposition}\label{prop:virtuallypolycyclicGroupsInjectiveIFFinjectiveandAutomorphismIFFAutomorphism}
	Let \(\Gamma\) be a finitely generated virtually solvable minimax group with a \(\mathbb Q\)-algebraic hull \(\HH\) whose identity component \(\HH^\circ\) has no non-trivial abelian \(\mathbb Q\)-factor. Assume \(\varphi = \theta \cdot \zeta\) for some endomorphisms \(\varphi, \theta, \zeta \colon \Gamma \to \Gamma\), where \(\zeta\) is a central endomorphism, then: 
	\begin{enumerate}[label=\emph{(\arabic*)}, ref={(\arabic*)}]
		\item 	 \(\varphi\) is injective if and only if \(\theta\) is injective.  \label{item:prop:virtuallypolycyclicGroupsInjectiveIFFinjectiveandAutomorphismIFFAutomorphism(1)}
		\item \(\varphi\) is an automorphism if and only if \(\theta\) is an automorphism.  \label{item:prop:virtuallypolycyclicGroupsInjectiveIFFinjectiveandAutomorphismIFFAutomorphism(2)}
	\end{enumerate}
\end{proposition}
\begin{proof}	Since \(x \mapsto \zeta(x)^{-1}\) defines a central endomorphism of \(\Gamma\) as well, it suffices to prove the two converse implications.
	
	\medskip 
	
We first reduce to the case where \(\HH\) is Zariski connected. 
	Set \(\Gamma^\circ := \Gamma \cap \HH^\circ\). Then \(\Gamma^\circ \lhd \Gamma\) has finite index, and \(\HH^\circ\) is the \(\mathbb Q\)-algebraic hull of \(\Gamma^\circ\) by Lemma~\ref{lemma:QalgebraicHullOfAFiniteIndexSubgroup}. 
	Moreover, \[\zeta(\Gamma)\subset Z(\Gamma)=\Gamma\cap Z(\HH)\subset \U(\HH)\subset \HH^\circ,\] so \(\zeta(\Gamma)\leq \Gamma^\circ\). If \(\theta\) is injective, let \(\Theta\) be the induced \(\mathbb Q\)-automorphism of \(\HH\). Then \(\Theta(\HH^\circ)=\HH^\circ\), so \(\theta(\Gamma^\circ)\leq \Gamma^\circ\). It follows that \(\varphi,\theta,\zeta\) restrict to endomorphisms of \(\Gamma^\circ\), and that \(\varphi\) and \(\theta\) induce the same endomorphism of the finite quotient \(\Gamma/\Gamma^\circ\). Consequently, once the proposition is proved for the connected pair \((\Gamma^\circ,\HH^\circ)\), the general case follows by passing to the finite quotient \(\Gamma/\Gamma^\circ\): injectivity of \(\varphi\) (and bijectivity of \(\varphi\) for \ref{item:prop:virtuallypolycyclicGroupsInjectiveIFFinjectiveandAutomorphismIFFAutomorphism(2)}) reduce to the corresponding statements on \(\Gamma^\circ\).
	
		\medskip 

Henceforth assume that \(\HH = \HH^\circ\) is Zariski connected.

\noindent \ref{item:prop:virtuallypolycyclicGroupsInjectiveIFFinjectiveandAutomorphismIFFAutomorphism(1)}. Assume that \(\theta\) is injective. Then \(\theta\) extends to a \(\mathbb Q\)-automorphism \(\Theta \colon \HH \to \HH\). Similarly, we may view \(\zeta\colon \Gamma \to \zeta(\Gamma)\) as a surjective homomorphism between finitely generated virtually solvable minimax groups. 
	Since \(\zeta(\Gamma) \leq Z(\Gamma) = \Gamma \cap Z(\HH) \subset \U\big(\HH\big),\) the \(\mathbb Q\)-algebraic hull of \(\zeta(\Gamma)\) equals the Zariski closure \(\mathbf Z\) of \(\zeta(\Gamma)\) inside \(\U(\HH)\).
	By \cite[Proposition 3.7]{DP26}, \(\zeta\) extends to a \(\mathbb Q\)-morphism \(\Psi \colon \HH \to \mathbf Z\). Since \(Z(\Gamma) \leq Z(\HH)\), we can view the \(\mathbb Q\)-morphism \(\Psi\) as a central \(\mathbb Q\)-endomorphism of \(\HH\).
	
	Next, define the \(\mathbb Q\)-endomorphism \(\Phi := \Theta \cdot \Psi\) of \(\HH\); it is clear that \(\Phi\) extends \(\varphi\colon \Gamma \to \Gamma\). On the Lie algebra \(\mathfrak h\), the inclusion \(Z\big(\mathfrak h\big) \subset [\mathfrak h, \mathfrak h] \) holds by  Proposition~\ref{prop:K-indecomposableGroupsIFFK-algebraicallyIndecomposableLiealgebras}. Now \(\Phi_\ast = \Theta_\ast + \Psi_\ast\), where \(\Theta_\ast\) is an automorphism and \(\Psi_\ast\) is a central endomorphism of \(\mathfrak h\). Therefore, 
	a similar argument to that in Proposition~\ref{prop:addendumtothm:CharacterizationAlgebraicAutomorphismsDirectSumLieAlgebras} implies that \(\Phi_\ast\) is an automorphism. By Proposition~\ref{prop:K-automorphismOfLieAlgebraIFFK-automorphismOfGroup}, \(\Phi\) is a \(\mathbb Q\)-automorphism of \(\HH\). Therefore, \(\varphi\colon \Gamma \to \Gamma\) is injective. 
	
	\medskip 
	
\noindent \ref{item:prop:virtuallypolycyclicGroupsInjectiveIFFinjectiveandAutomorphismIFFAutomorphism(2)}. Assume moreover that \(\theta\) is surjective, and keep the same notations as above. Define \(\Delta:= \Theta^{-1} \circ \Psi \colon \HH \to Z(\HH)\). By Proposition~\ref{prop:K-indecomposableGroupsIFFK-algebraicallyIndecomposableLiealgebras}, it holds that \(Z(\HH) = Z(\HH)^\circ \subset [\HH,\HH]\), hence
\[\Delta^2(\HH) \subset \Delta([\HH,\HH]) = \{e\}\] by centrality of \(\Delta\). Next, define the \(\mathbb Q\)-automorphism \(E \in \operatorname{Aut}(\HH)\) by \(E(h) = h \cdot \Delta(h)\) with inverse \(E^{-1}(h) = h \cdot \Delta(h)^{-1}\). It is easy to see that \(\Phi = \Theta \circ E\). 

Since \(\theta(\Gamma) = \Gamma\), we have that \(\Theta^{\pm 1}(\Gamma) = \Gamma\). Hence \(\Delta(\Gamma) = \Theta^{-1}\big( \Psi(\Gamma) \big) \leq \Gamma\). It follows that \(E^{\pm 1}(\Gamma) \leq \Gamma\), and thus \(E^{\pm 1}(\Gamma) = \Gamma\). Therefore, \(E{|_\Gamma}\) is an automorphism of \(\Gamma\), and so is \[
\varphi = \Phi|_{\Gamma} = (\Theta \circ E)|_{\Gamma} = \theta \circ E|_{\Gamma}.  
\]  
\end{proof}

\noindent  

We now apply the \CMP{}-statement Theorem~\ref{thm:virtuallysolvablegroups} for virtually solvable \(\mathbb Q\)-groups with strong unipotent radical.  The following lemma allows us to reduce to the case where the identity component \(\HH_i^\circ\) of each factor \(\Gamma_i\) is non-abelian: this case is very rigid. 
\begin{lemma}\label{lemma:infinitedihedral}
	Let \(\Gamma\) be a non-abelian finitely generated virtually solvable minimax group. Assume \(\Gamma\) has a \(\mathbb Q\)-algebraic hull  \(\HH\) whose identity component \(\HH^\circ\) is abelian and \(\mathbb Q\)-indecomposable. Then \(\Gamma\) is isomorphic to the infinite dihedral group \(\mathcal D_{\infty} := \mathbb Z \rtimes_{-1} C_2\) and \(\HH \cong \mathbb G_{\mathrm{a}} \rtimes_{-1} C_2\). 
\end{lemma}
\begin{proof}
Let \(\U:= \U(\HH) = \U(\HH^\circ)\) be the unipotent radical. Then \(\T := \HH^\circ / \U\) is a torus, and when viewed inside the abelian \(\mathbb Q\)-group \(\HH^\circ\) via a Levi decomposition, centralizes \(\U\). Hence \(\T = 1\) and \(\HH^\circ = \U = \mathbb{G}_{\mathrm{a}}\) by \(\mathbb Q\)-indecomposability.

 But the strong unipotent radical property implies the action of \(\HH / \HH^\circ\) on \(\HH^\circ = \U\cong \mathbb{G}_{\mathrm{a}}\) is faithful. As the only non-trivial finite subgroup of \(\operatorname{Aut}_{\mathbb Q}(\mathbb{G}_{\mathrm{a}}) = \mathbb Q^\times\) is \(\{\pm 1\}\),  \(\HH / \HH^\circ\)  has order~\(\leq\,\)2. Since \(\Gamma\) is  non-abelian, \(\HH / \HH^\circ \cong C_2\), with \(C_2\) acting by inversion on \(\HH^\circ \cong \mathbb{G}_{\mathrm{a}}\). Therefore \(\HH \cong \mathbb{G}_{\mathrm{a}} \rtimes_{-1} C_2\). 
 
 It is straightforward to show that \(\Gamma^\circ = \Gamma \cap \HH^\circ\) is infinite cyclic, \([\Gamma \colon \Gamma^\circ]=2\), and  \(\Gamma \cong \mathbb{Z} \rtimes_{-1} C_2\). 
\end{proof}

\noindent If \(\Gamma = \Gamma_1\times \cdots \times \Gamma_r\) is a direct product of finitely generated virtually solvable minimax groups where each factor \(\Gamma_i\) has a \(\mathbb Q\)-algebraic hull \(\HH_i\), then \(\HH = \HH_1  \times \cdots \times \HH_r\) is a \(\mathbb Q\)-algebraic hull of \(\Gamma\). 

\begin{letterthmbody}\label{thm:CharacterizationAutomorphismDirectProductOfMinimaxGroups}
	Let \(\Gamma = \Gamma_1  \times \cdots \times \Gamma_r\) be a direct product of non-abelian finitely generated virtually solvable minimax groups having \(\mathbb Q\)-algebraic hulls \(\HH_i\) with \(\mathbb Q\)-indecomposable identity components \(\HH_i^\circ\).  
	Then every monomorphism \(\varphi\colon \Gamma \to \Gamma\) uniquely factorizes as
	\[
	\varphi = \theta \cdot \zeta,
	\]
	where:  
	\begin{enumerate}[label=\emph{(\arabic*)}, ref={(\arabic*)}]
		\item \(\theta\) is a monomorphism of \(\Gamma\) and there is a unique permutation \(\sigma \in \mathcal S_r\) such that \(\theta\) sends each factor \(\Gamma_i\) into a factor \(\Gamma_{\sigma(i)}\), where \(\Gamma_i\) and \(\Gamma_{\sigma(i)}\) have \(\mathbb Q\)-isomorphic \(\mathbb Q\)-algebraic hulls.   \label{item:thmCharacterizationAutomorphismDirectProductOfMinimaxGroups(1)} 
		\item \(\zeta\) is a central endomorphism of \(\Gamma\) such that \(\pi_{\sigma(i)}\big(\zeta(\Gamma_i)\big) = \{e\}\).  \label{item:thmCharacterizationAutomorphismDirectProductOfMinimaxGroups(2)} 
	\end{enumerate}
	Conversely, every such pair \((\theta, 
\zeta)\) defines a monomorphism \(\varphi:= \theta \cdot \zeta\) of \(\Gamma\).

\medskip 

\noindent 	Moreover, if \(\varphi \in \operatorname{Aut}(\Gamma)\), then \(\theta \in \operatorname{Aut}(\Gamma)\) sends every factor \(\Gamma_i\) onto an isomorphic factor \(\Gamma_{\sigma(i)}\). 
\end{letterthmbody}
\begin{proof} 
	We first reduce to the case where every \(\HH_i^\circ\) is non-abelian. Indeed, by Lemma~\ref{lemma:infinitedihedral}, any factor \(\Gamma_i\) with abelian identity component \(\HH_i^\circ\) is necessarily \(\Gamma_i\cong \mathcal D_\infty\) with hull \(\HH_i\cong \mathbb G_{\mathrm a}\rtimes_{-1} C_2\). These dihedral factors form a rigid centerless block: a routine verification using Lemma~\ref{lemma:componentMorphisms} shows that for any \(\mathbb Q\)-automorphism \(\Phi\) of \(\HH\), all component maps between this block and the complementary block with non-abelian \(\HH_i^\circ\) vanish, so the product splits into two independent blocks. The case where each factor \(\Gamma_i\) is infinite dihedral follows directly from the centerlessness of the hulls \(\HH_i\). Henceforth, we assume that all \(\HH_i^\circ\) are non-abelian.
	
	\medskip

	Using Theorem \ref{thm:KalgebraicKrullSchmidt} and Proposition~\ref{prop:K-indecomposableGroupsIFFK-algebraicallyIndecomposableLiealgebras}, we deduce that \(\HH^\circ = \HH_1^\circ \times \cdots \times \HH_r^\circ\) has no non-trivial abelian \(\mathbb Q\)-factor. Let \(\varphi\colon \Gamma \to \Gamma\) be a monomorphism and let \(\Phi\) be the induced \(\mathbb Q\)-automorphism of \(\HH\).  By Theorem \ref{thm:virtuallysolvablegroups}, \(\Phi\) decomposes as \(\Phi = \Theta \cdot \Psi\), where \(
\Theta\) is an automorphism of \(\HH\) that maps each factor \(\HH_i\) onto an isomorphic factor \(\HH_{\sigma(i)}\) for some unique \(\sigma \in \mathcal S_r\),  and \(\Psi\) is a central endomorphism of \(\HH\) satisfying \(\pi_{\sigma(i)}\big(\Psi\big(\HH_i\big)\big) = 
\{e\}\) for each \(i\). 

 Using the matrix of component maps (Lemma~\ref{lemma:componentMorphisms}), it is easy to see that \(\Theta(\Gamma) \leq \Gamma\) and \(\Psi(\Gamma) \leq \Gamma\). Thus, we may define \(\theta, \zeta \colon \Gamma \to \Gamma\) by restricting \(\Theta, \Psi\). Then \ref{item:thmCharacterizationAutomorphismDirectProductOfMinimaxGroups(1)} and \ref{item:thmCharacterizationAutomorphismDirectProductOfMinimaxGroups(2)} immediately follow. 
  Uniqueness of the decomposition \(\varphi = \theta \cdot \zeta\) is immediate again from the matrix of component maps. 
 
 Assume now that \(\varphi\) is an automorphism.  Then Proposition  \ref{prop:virtuallypolycyclicGroupsInjectiveIFFinjectiveandAutomorphismIFFAutomorphism} yields that \(\theta\) is an automorphism as well, mapping every factor \(\Gamma_i\) onto an isomorphic factor \(\Gamma_{\sigma(i)}\). 
 
 The converse statement is an immediate consequence of Proposition \ref{prop:virtuallypolycyclicGroupsInjectiveIFFinjectiveandAutomorphismIFFAutomorphism}. 
\end{proof}

\subsection{The \texorpdfstring{\(\mathbb Q\)}{Q}-indecomposability assumption}\label{subsec:assumptionsDECOMP}

In this section, we give counterexamples showing that the \(\mathbb Q\)-indecomposability assumption in Theorem~\ref{thm:CharacterizationAutomorphismDirectProductOfMinimaxGroups} cannot be weakened to direct indecomposability of the factors. More precisely, we construct directly indecomposable groups \(\Gamma\) for which the identity component \(\HH^\circ\) is \(\mathbb Q\)-decomposable, and such that both the forward and converse directions of \CMP{} on automorphisms fail for the direct product \(\Gamma \times \Gamma\).

This already happens in the nilpotent case. In particular, the indecomposability assumption on the rational Mal'cev completion in \cite[Theorem~3.1]{Sen24} (``nilpotent Theorem~\ref{thm:CharacterizationAutomorphismDirectProductOfMinimaxGroups}'') cannot be weakened to direct indecomposability of the factors.

\begin{example} \label{example:Q-indecomposableNILPOTENT}
	\textit{We construct a finitely generated torsion-free nilpotent group \(\Gamma\) which is directly indecomposable and whose rational Mal'cev completion \(\Gamma^\mathbb Q\) is decomposable (equivalently, its unipotent hull \(\HH\) is \(\mathbb Q\)-decomposable). For the direct product \(\Gamma\times\Gamma\), we then construct counterexamples to both directions of \CMP{} on automorphisms.}

	The group construction is due to Gilbert Baumslag \cite[\S 5]{Bau75} (see also \cite[\S 5]{BMO16} and \cite[p. 529]{Sen24}). 
	Fix an integer \(p \geq 2\). Let \(\langle a,b,c\rangle \cong \mathbb Z^3\) and  \(\langle t \rangle \cong \mathbb Z\). Define the semidirect product
	\[
	B = \mathbb Z^3 \rtimes_A \mathbb Z = \big\langle a,b,c,t \mid t^{-1}at = ab, t^{-1}bt = bc, \text{\(c\) central} \big\rangle, \qquad A = \textstyle{\begin{bmatrix}
			1&0&0\\-1&1&0\\1&-1&1
		\end{bmatrix}.}  
	\]
	Then \(B\) is torsion-free nilpotent of class 3. Set \(K := B \times F\) where \(F = \langle f \rangle \cong \mathbb Z\); then \(K\) is torsion-free nilpotent of class 3 as well.
	
	Let \(K^\mathbb Q  = \HH_K(\mathbb Q)\) be the rational Mal'cev completion of \(K\), and let \(s \in K^\mathbb Q\) be the unique \(p^\mathrm{th}\)~root of \(bf\). Set \(\Gamma := \langle K, s\rangle\). Then \(\Gamma\) is finitely generated torsion-free nilpotent, and \(K\) has finite index in \(\Gamma\) since \(K \subset \Gamma \subset K^\mathbb Q\). In fact, one computes that \([\Gamma : K] = p^2\). As unipotent hulls are Zariski connected, Lemma~\ref{lemma:QalgebraicHullOfAFiniteIndexSubgroup} yields \(\HH_{\Gamma} = \HH_{K}\), and hence \(\Gamma^\mathbb Q = K^\mathbb Q\).

	In \cite[Lemma 3]{Bau75} it is shown that \(\Gamma\) is directly indecomposable. Since \(K = B \times F\), we have \(\HH_\Gamma = \HH_K = \HH_B \times \mathbb{G}_\mathrm{a}\), so \(\HH_\Gamma\) is \(\mathbb Q\)-decomposable. Moreover, a computation shows that \([t,s^{-1}] = c^{1/p}\), hence \(c^{1/p} \in \Gamma\).
	
	For ease of computation, we further mention that \((t,a,b,c,f)\) is a \emph{Mal'cev basis} for \(K\) and \((t,a,s,c^{1/p},f)\) is a Mal'cev basis for \(\Gamma\) (see e.g.\ \cite[\S2]{Dek18}).  In the associated Mal'cev normal forms, multiplication is given by 
	\begin{align*}
		&t^{\tau_1} a^{\alpha_1} b^{\beta_1} c^{\gamma_1} f^{\phi_1}  \cdot 	t^{\tau_2} a^{\alpha_2} b^{\beta_2} c^{\gamma_2} f^{\phi_2}  = t^{\tau_1 + \tau_2} a^{\alpha_1 + \alpha_2} b^{\beta_1 + \beta_2 + \tau_2 \alpha_1} c^{\gamma_1 + \gamma_2 + \tau_2 \beta_1 + \binom{\tau_2}2 \alpha_1} f^{\phi_1+\phi_2}, \\
		&t^{\tau_1} a^{\alpha_1} s^{\sigma_1} c^{\gamma_1/p} f^{\phi_1}  \cdot 	t^{\tau_2} a^{\alpha_2} s^{\sigma_2} c^{\gamma_2/p} f^{\phi_2} = t^{\tau_1 + \tau_2} a^{\alpha_1 + \alpha_2} s^{\sigma_1 + \sigma_2 + p \tau_2 \alpha_1} c^{\left(\gamma_1 + \gamma_2 + \tau_2 \sigma_1 + p \binom{\tau_2}2 \alpha_1\right)_{\big/ \text{\small $p$}}} f^{\phi_1+\phi_2 -  \tau_2 \alpha_1},
	\end{align*}
	where \(\binom n 2 = n(n-1)/2\). In particular, \(Z(\Gamma) = \langle c^{1/p} ,f \rangle\).

Define \(\lambda\colon \Gamma \to \mathbb Z\) by \(\lambda\big(t^{\tau} a^{\alpha} s^{\sigma} c^{\gamma/p} f^{\phi} \big) = p \phi + \sigma\) and define \(\zeta_0\colon \Gamma \to Z(\Gamma)\) by \(\zeta_0(x) = f^{\lambda(x)}\). 

\medskip 	
	
\noindent \textit{1. Forward \CMP{}-direction fails.} Define the endomorphism \(\psi \in \operatorname{End}(\Gamma \times \Gamma)\) by
	\[
\psi(x,y) = \big(x \zeta_0(x) \, \zeta_0(y), y \zeta_0(y)^{-1} \zeta_0(x)^{-1}  \big) = \big( x f^{\lambda(x) + \lambda(y)}, y f^{-\lambda(x) - \lambda(y)}  \big). 
	\]
Using that the quantity \(\Lambda(x,y)=\lambda(x)+\lambda(y)\) is preserved by \(\psi\), we see that \(\psi\) is an automorphism with inverse \(\psi^{-1}(x,y) = \big( x f^{-\lambda(x) - \lambda(y)}, y f^{\lambda(x) + \lambda(y)}  \big)\). However, using that \(x \mapsto x f^{\lambda(x)}\) is \emph{not} an automorphism of \(\Gamma\), for it induces the map 
\[
\begin{bmatrix}
	1 & 0 \\
	0 & p+1
	\end{bmatrix}
\] on the center \(Z(\Gamma) = \langle c^{1/p},f\rangle\), the automorphism \(\psi\) is of neither component map form 
\[
\begin{bmatrix}
	\theta_1 & \zeta_1 \\
	\zeta_2 & \theta_2
\end{bmatrix}, \qquad \begin{bmatrix}
 \zeta_1 & \theta_1  \\
 \theta_2 & \zeta_2
\end{bmatrix}
\]
for any automorphisms \(\theta_1,\theta_2\colon \Gamma \to \Gamma\) or central endomorphisms \(\zeta_1,\zeta_2\colon \Gamma \to Z(\Gamma)\). 
\medskip 
	
\noindent \textit{2. Converse \CMP{}-direction fails.} Define the endomorphism \(\varphi \in \operatorname{End}(\Gamma \times \Gamma)\) by \[ \varphi(x,y) = \big(x \, \zeta_0(y), \zeta_0(x) \, y \big).\] 
	Then \(\varphi\) is of the converse-\CMP{} form \(\theta \cdot \zeta\), but \(\varphi\) is not surjective: for instance \((f,1) \not\in \varphi(\Gamma \times \Gamma)\).  \demo
\end{example}

We now construct an analogous non-nilpotent torsion-free polycyclic counterexample. In addition to extending the preceding phenomenon beyond the nilpotent case, this example will play a second role in Subsection~\ref{subsec:twistedConjugacy}, where it yields a counterexample to a Reidemeister number formula.

\begin{example}\label{example:Q-indecomposabilityPolycyclic}
		Consider the torsion-free polycyclic group  
		\[
		\Gamma = \mathbb Z^4 \rtimes_A \mathbb Z, \qquad A = 
		\begin{bmatrix}
			1 & 0 & 1 & 0\\
			0 & 1 & 0 & 1\\ 
			0 & 0 & 3 & 2\\
			0 & 0 & 4 & 3
		\end{bmatrix}
		= \begin{bmatrix}
			I_2 & I_2 \\ 0 & B
		\end{bmatrix}, \qquad B = \begin{bmatrix}
			3 & 2 \\ 4 & 3
		\end{bmatrix}.
		\] 
		Observe that \(A\) is semisimple (i.e.\ diagonalizable over \(\mathbb C\)). We compute that \[
		F := \operatorname{Fitt}(\Gamma)  = \big\{ (v,0)  \mid v \in \mathbb Z^4 \big\} = \mathbb Z^4, \qquad 
		Z(\Gamma)  = \left\{\left(v_1,v_2,0,0;0\right) \, \middle| \, v_1,v_2 \in \mathbb Z\right\} = \mathbb Z^2.
		\]
	Below we prove \(\Gamma\) is directly indecomposable, and we construct its \(\mathbb Q\)-algebraic hull \(\HH\). The resulting hull is Zariski connected but \(\mathbb Q\)-decomposable, and the computation illustrates the procedure in a concrete polycyclic example. First, we exhibit the failure of \CMP{} on automorphisms of \(\Gamma \times \Gamma\).

	Define the central endomorphism \(\zeta_0\colon \Gamma \to Z(\Gamma)\) by
	\[
	\zeta_0(v,k)=(Lv,0),
	\qquad
	L=
	\begin{bmatrix}
		-2 & -2 & 1 & 0\\
		0 & 0 & 0 & 0\\
		0 & 0 & 0 & 0\\
		0 & 0 & 0 & 0
	\end{bmatrix}.
	\]
	Since \(LA=L\), we have \(LA^m=L\) for all \(m\in \mathbb Z\), so \(\zeta_0\colon \Gamma \to Z(\Gamma)\) is a homomorphism.
	
	\medskip
	
	\noindent  \emph{1. Forward \CMP{}-direction fails.}
	Define \(\psi\in \operatorname{End}(\Gamma\times \Gamma)\) by
	\[
	\psi(x,y)=\bigl(x\,\zeta_0(x)\zeta_0(y),\, y\,\zeta_0(x)^{-1}\zeta_0(y)^{-1}\bigr).
	\]
	 We claim that \(\psi\) is an automorphism. Indeed, \(\psi\) induces the identity on
	\[
	(\Gamma\times \Gamma)/(F\times F)\cong (\Gamma/F)\times (\Gamma/F)\cong \mathbb Z^2,
	\]
	and on \(F \times F \cong \mathbb Z^8 \) it is represented by  
	\[
	\psi|_{F\times F}
	=
	\begin{bmatrix}
		I_4+L & L\\
		-L & I_4-L
	\end{bmatrix}.
	\]
	A direct computation shows that \(
	\det(\psi|_{F\times F})=1,
	\) 
	so \(\psi|_{F\times F}\) is an automorphism of \(F\times F\). Therefore \(\psi\) is an automorphism of \(\Gamma\times\Gamma\).  Since \(y\mapsto y\,\zeta_0(y)^{-1}\) is not an automorphism, the same argument as in Example~\ref{example:Q-indecomposableNILPOTENT} shows that \(\psi\) is not of the prescribed \CMP{} form.

	\medskip
	
	\noindent  \emph{2. Converse \CMP{}-direction fails.}
	Define the endomorphism \(\varphi \in \operatorname{End}(\Gamma \times \Gamma)\) by \[\varphi(x,y) = \big(x \, \zeta_0(y), \zeta_0(x) \, y\big).\] 
	We claim \(\varphi\) is not an automorphism of \(\Gamma \times \Gamma\).  If \(\varphi\) were an automorphism, then it would restrict to an automorphism of 
	\(
	\operatorname{Fitt}(\Gamma \times \Gamma) = F \times F \cong \mathbb Z^8.
	\)
	On \(F \times F\) however, it is represented by 
	\[
	\varphi|_{F \times F} = \begin{bmatrix}
		I_4 & L \\ L & I_4
	\end{bmatrix},
	\] and one computes \(\det(\varphi|_{F \times F}) = -3 \neq \pm 1\), a contradiction.

	\medskip 
	
	\noindent \textit{\(\Gamma\) is directly indecomposable.} Suppose, for a contradiction, that \(\Gamma = \Gamma_1 \times \Gamma_2\) with \(\Gamma_1,\Gamma_2 \neq 1\).  Since \(F = \operatorname{Fitt}(\Gamma) = \operatorname{Fitt}(\Gamma_1) \times \operatorname{Fitt}(\Gamma_2)\), we have  
	\[
	{\Gamma}/{F} \cong \big({\Gamma_1}/{\operatorname{Fitt}(\Gamma_1)}\big) \times \big({\Gamma_2}/{\operatorname{Fitt}(\Gamma_2)}\big).
	\]
	As \(\Gamma / F \cong \mathbb Z\) is directly indecomposable, one of the two quotients must be trivial. After swapping indices if necessary, assume \(\Gamma_1 = \operatorname{Fitt}(\Gamma_1)\). Then \(\Gamma_1 \lhd \Gamma\) is nilpotent, hence \(\Gamma_1 \leq F \cong \mathbb Z^4\) is free abelian.  In particular, \(\Gamma_1\) has an infinite cyclic direct factor \(C \cong \mathbb Z\), so \(\Gamma\) admits a decomposition 
	\(
	\Gamma = C \times H
	\)
	for some subgroup \(H \leq \Gamma\). Notice that \(C \leq Z(\Gamma)\).
	
	Choose a homomorphism \(f \colon \Gamma \to \mathbb Z\) such that \(f|_C\colon C \to \mathbb Z\) is an isomorphism, and set \(\ell:= f|_F \in \operatorname{Hom}(\mathbb Z^4, \mathbb Z)\). Since \(f(tvt^{-1}) = f(v)\) for all \(v \in F\), we have  \(\ell A = \ell\). Solving \(\ell (A - I_4) = 0\) gives
	\[ \ell = \begin{bmatrix}
		2m & 2n & m-2n & -m + n
	\end{bmatrix} \qquad \text{for some \(m,n\in \mathbb Z\)}. \]
	In particular, \(\ell\) takes even values on \(Z(\Gamma)\), hence \(\ell(C) \leq 2\mathbb Z\). On the other hand, \(\ell(C) = f(C) = \mathbb Z\) (since \(C \leq F\)), a contradiction.
	
	\medskip 
	
	\noindent \textit{{\(\mathbb Q\)-algebraic hull of \(\Gamma\)}.} 
	Following \cite[Proof of Theorem 3.5]{DP26}, we construct the \(\mathbb Q\)-algebraic hull \(\HH\) of \(\Gamma\). Consider the faithful representation \(\rho\colon \Gamma \to \operatorname{GL}_6(\mathbb Z)\) given by 
	\[
	\rho(v,k) = \begin{bmatrix}
		A^k & 0 & v \\
		0 & 1 & k \\
		0 & 0 & 1
	\end{bmatrix}.
	\] 
	Let \(\HH \) be the Zariski closure of \(\rho(\Gamma)\) in \(\operatorname{GL}_6(\mathbb C)\). Since \(\rho(\Gamma)\subset \operatorname{GL}_6(\mathbb Q)\), it follows from \cite[Proposition 1.3(b)]{Bor91} that \(\HH\) is a \(\mathbb Q\)-defined linear algebraic group. We show that \(\HH\) is the \(\mathbb Q\)-algebraic hull of \(\Gamma\), along the way determining its full structure. 
	\medskip 
	
	\begin{itemize}[leftmargin=1.8em]
		\item \textit{Unipotent radical.}  Set \[\U := \left\{ \begin{bmatrix}
			I_4 & 0 & x \\ 0 & 1 & s \\ 0 & 0 & 1
		\end{bmatrix} \,\middle|\,  x \in \mathbb C^4, s \in \mathbb C \right\} \cong \mathbb G_{\mathrm{a}}^5 .\] We claim that \(\U = \U(\HH)\) equals the unipotent radical of \(\HH\). 
		
		For \(v\in \mathbb Z^4\), the element \(\rho(v,0) \in \HH\) is unipotent, and taking the Zariski closure over \(v \in \mathbb Z^4\) yields 
		\[\U_1 := \left\{ \begin{bmatrix}
			I_4 & 0 & x \\ 0 & 1 & 0 \\ 0 & 0 & 1
		\end{bmatrix} \,\middle|\,  x \in \mathbb C^4 \right\} \cong \mathbb{G}_{\mathrm a}^4. \]
		Moreover, for \(k \in \mathbb Z\), we have the multiplicative Jordan decomposition
		\begin{equation*}
			\rho(0,k) = \underbrace{\begin{bmatrix}
					A^k & 0 & 0 \\ 0 & 1 & 0 \\ 0 & 0 & 1
			\end{bmatrix}}_{\text{semisimple}} \cdot \underbrace{\begin{bmatrix}
					I_4 & 0 & 0 \\ 0 & 1 & k \\ 0 & 0 & 1
			\end{bmatrix}}_{\text{unipotent}}, 
		\end{equation*}
		so \(\rho(0,k)_u \in \HH\). Taking the Zariski closure over \(k \in \mathbb Z\) yields 
		\[
		\U_2 := \left\{ \begin{bmatrix}
			I_4 & 0 & 0 \\ 0 & 1 & s \\ 0 & 0 & 1
		\end{bmatrix} \,\middle|\,  s \in \mathbb C \right\} \cong \mathbb{G}_{\mathrm a}. \] 
		The subgroups \(\U_1\) and \(\U_2\) commute, hence \(\U = \U_1 \cdot \U_2\) is a connected unipotent subgroup of \(\HH\). Since \(\rho(\Gamma)\) normalizes \(\U\), Zariski density implies \(\U \lhd \HH\), and therefore \(\U \subset \U(\HH)\). By \cite[Proposition 3.2]{DP26}, \(\dim \U(\HH) \leq h(\Gamma) = 5\) and hence \(\U = \U(\HH)\).
		
		\item \textit{Strong unipotent radical.} We have \[
		\HH \subset \left\{ \begin{bmatrix}
			T  & 0 & x \\ 0 & 1 & s \\ 0 & 0 & 1
		\end{bmatrix} \in \operatorname{GL}_6(\mathbb C) \,\middle|\, T\in \operatorname{GL}_4(\mathbb C), x \in \mathbb C^4 , s \in \mathbb C\right\}, 
		\]
		since the right-hand side is Zariski closed and contains \(\rho(\Gamma)\).  A direct computation then shows \(C_{\HH}(\U) \subset \U\). Hence \(\HH\) has a strong unipotent radical, and \(\HH\) satisfies the conditions of Definition~\ref{def:Qalgebraichull}. Thus \(\HH\) is the \(\mathbb Q\)-algebraic hull of \(\Gamma\).  
		
		\item \textit{Torus, connectedness and \(\mathbb Q\)-decomposability.} To describe \(\HH\) more explicitly, we show that \(\HH / \U\) is a 1-dimensional torus. It will then follow that \(\HH\) is Zariski connected and \(\mathbb Q\)-decomposable. 
		
		First, let \(\T_A\) be the Zariski closure of \(\langle A \rangle \cong \mathbb Z\) in \(\operatorname{GL}_4(\mathbb C)\). The matrix \(A\) diagonalizes as
		\[
		P^{-1} A P = \operatorname{Diag}(1,1,\lambda,\lambda^{-1})
		\]
		for some \(P \in \operatorname{GL}_4(\mathbb C)\), where \(\lambda^{\pm 1} = 3 \pm 2 \sqrt 2\) are the eigenvalues of \(B\). 
		Therefore, 
		\[
		\T_A = \overline{\big\{ P\operatorname{Diag}(1,1,\lambda^k,\lambda^{-k})P^{-1} \mid k \in \mathbb Z \big\}}^{\mathrm{Zar}}=  \big\{ P\operatorname{Diag}(1,1,z,z^{-1})P^{-1} \mid z \in \mathbb C^\times \big\}  \cong \mathbb G_{\mathrm{m}},
		\]
		a 1-dimensional torus, hence connected.  
		Accordingly, define 
		\[
		\widetilde{\T}_A := \left\{\begin{bmatrix}
			T & 0 & 0 \\ 0 & 1 & 0 \\ 0 & 0 & 1
		\end{bmatrix} \in \operatorname{GL}_6(\mathbb C)\, \middle| \, T \in \T_A\right\} \subset  \HH.\]
		Equivalently, \(\widetilde{\T}_A \cong \T_A\) is the Zariski closure of \(\{\rho(0,k)_s = \operatorname{Diag}(A^k,1,1) \mid k \in \mathbb Z \} \subset \HH\).
		
		Next, one computes 
		\[
		\widetilde \T_A \cdot \U = \left\{ \begin{bmatrix}
			T & 0 & x \\ 0 & 1 & s \\ 0 & 0 & 1
		\end{bmatrix} \,\middle|\, T \in \T_A, x \in \mathbb C^4, s \in \mathbb C\right\} \subset \HH.
		\]
		Since the left-hand side of this inclusion is Zariski closed and contains \(\rho(\Gamma)\), equality holds by Zariski density.  Therefore \(\HH = \U \rtimes \widetilde{\T}_{A}\cong \mathbb G_{\mathrm{a}}^5 \rtimes \T_A \), and \(\HH\) is Zariski connected. 
		
		Finally, the factorization
		\[
		\begin{bmatrix}
			T & 0 & x \\ 0 & 1 & s \\ 0 & 0 & 1
		\end{bmatrix}  = \begin{bmatrix}
			I_4 & 0 & 0 \\ 0 & 1 & s \\ 0 & 0 & 1
		\end{bmatrix} \cdot  \begin{bmatrix}
			T & 0 & x \\ 0 & 1 & 0 \\ 0 & 0 & 1
		\end{bmatrix}
		\]
		shows that \(\HH = \mathbb G_{\mathrm{a}} \times \big( \mathbb G_{\mathrm{a}}^4 \rtimes \T_A \big)\) as a  \(\mathbb Q\)-group, hence \(\HH = \HH^\circ\) is \(\mathbb Q\)-decomposable.  \demo 
	\end{itemize}	
\end{example}

	\subsection{Two explicit non-nilpotent polycyclic examples} \label{subsec:example}
If \(\Gamma\) is finitely generated nilpotent, then it is \textsf{WTN} if and only if it is torsion-free. In that case, \(\Gamma\) admits a unipotent hull \(\HH\), see Example~\ref{example:unipotenthull}. We now give two explicit examples of non-nilpotent torsion-free polycyclic groups satisfying the assumptions of Theorem~\ref{thm:CharacterizationAutomorphismDirectProductOfMinimaxGroups}.

The first example is an elementary example of the form \(\Gamma = \mathbb Z^4 \rtimes_A \mathbb Z\) with \(A\) non-unipotent. It gives a basic non-nilpotent example with a Zariski connected, \(\mathbb Q\)-indecomposable hull, yet with \(Z(\Gamma)=1.\) The second example is Heisenberg-by-cyclic and has a non-trivial center. It is essentially of the form
\[
\Gamma = H_5(\mathbb Z) \rtimes \mathbb Z,
\]
where \(H_5(\mathbb Z)\) is the rank-\(5\) Heisenberg group, and the generator of \(\mathbb Z\) acts on the abelianization \(H_5(\mathbb Z)^{\mathrm{ab}} \cong \mathbb Z^4 \) by a non-unipotent automorphism and fixes the center \(Z(H_5(\mathbb Z))=\langle z\rangle\). We show that its \(\mathbb Q\)-algebraic hull is Zariski connected and \(\mathbb Q\)-indecomposable.

\begin{example}\label{example:minimalExample} We keep the details of the construction minimal. All omitted verifications are analogous to those in Example~\ref{example:Q-indecomposabilityPolycyclic}, which exhibits a Zariski connected but \(\mathbb Q\)-decomposable hull. 
	
	Consider the torsion-free polycyclic group
	\[
	\Gamma = \mathbb Z^4 \rtimes_A \mathbb Z, \qquad A = 
	\begin{bmatrix}
		2 & 1 & 2 & 1\\
		1 & 1 & 1 & 1\\ 
		0 & 0 & 2 & 1\\
		0 & 0 & 1 & 1
	\end{bmatrix}
	= \begin{bmatrix}
		B & B \\ 0 & B
	\end{bmatrix}, \qquad B = \begin{bmatrix}
		2 & 1 \\ 1 & 1
	\end{bmatrix}.
	\]
	We have \[
	F := \operatorname{Fitt}(\Gamma)  = \big\{ (v,0)  \mid v \in \mathbb Z^4 \big\} = \mathbb Z^4, \qquad 
	Z(\Gamma)  = 1.
	\]
	Consider the faithful representation \(\rho\colon \Gamma \to \operatorname{GL}_5(\mathbb C)\) given by 
	\[
	\rho(v,k) = \begin{bmatrix}
		A^k & v \\
		0 & 1
	\end{bmatrix}.
	\]
	Set \(\HH := \overline{\rho(\Gamma)}^{\mathrm{Zar}}\). Then \(\HH\) is the \(\mathbb Q\)-algebraic hull of \(\Gamma\). Its unipotent radical is
	\[
	\U = \left\{\begin{bmatrix}
		I_4 + sN & w \\ 0 & 1 
	\end{bmatrix} \:\middle| \: w\in \mathbb C^4, s \in \mathbb C\right\}, \qquad sN = \begin{bmatrix}
		0&sI_2\\ 0&0 
	\end{bmatrix},
	\]
	so \(\dim \U = h(\Gamma) = 5\). 
	Writing  
	\[
	\mathbf{V} = \left\{\begin{bmatrix}
		I_4  & w \\ 0 & 1 
	\end{bmatrix} \:\middle| \: w\in \mathbb C^4\right\}, \qquad \mathbf{U}_1 = \left\{\begin{bmatrix}
		I_4 + sN & 0 \\ 0 & 1 
	\end{bmatrix} \:\middle| \: s\in \mathbb C\right\},
	\]
	we have \(\U = \mathbf{V} \rtimes \U_1 \cong \mathbb G_{\mathrm{a}}^4 \rtimes \mathbb{G}_{\mathrm{a}}\). 
	
	The semisimple part of \(\rho(0,k)\) is \[
	\begin{bmatrix}
		B^k & 0 & 0\\
		0 & B^k & 0 \\
		0 & 0 & 1
	\end{bmatrix} = \begin{bmatrix}
		\widetilde{B}^k & 0 \\ 0 & 1
	\end{bmatrix}, \qquad \widetilde{B} = \operatorname{Diag}(B,B).
	\]
	Hence \[\T := \overline{\left\langle \begin{bmatrix}
			\widetilde B & 0 \\0 & 1
		\end{bmatrix}  \right\rangle}^{\mathrm{Zar}} = \left\{\begin{bmatrix}
		a I_4 + b C & 0 \\ 0 & 1
	\end{bmatrix} \in \operatorname{GL}_5(\mathbb C) \: \middle| \: a^2 -5b^2 = 1\right\}  \leq \HH, \qquad  C := 2 \widetilde{B} - 3I_4 \] is a 1-dimensional \(\mathbb Q\)-torus. In fact, \(\HH = \U \rtimes \T\), and over \(\mathbb C\) we have 
	\[
	\HH = \U \rtimes \T \cong (\mathbb{G}_{\mathrm{a}}^4 \rtimes \mathbb{G}_{\mathrm{a}}) \rtimes \mathbb{G}_{\mathrm{m}},
	\]
	so \(\HH\) is Zariski connected. 
	
	It remains to show that \(\HH\) is \(\mathbb Q\)-indecomposable. By Proposition~\ref{prop:K-indecomposableGroupsIFFK-algebraicallyIndecomposableLiealgebras}, it suffices to prove that \(\mathfrak h:=\operatorname{Lie}(\HH)\) is indecomposable.

	\medskip 
	
	\noindent \textit{Lie algebra \(\mathfrak h\)}. First, we compute
	\[
	\operatorname{Lie}(\U) = \left\{\begin{bmatrix}
		sN & w \\
		0 & 0
	\end{bmatrix} \:\middle|\: w\in \mathbb{C}^4, s\in \mathbb C\right\} \cong \mathbb{C}^4 \rtimes_N \mathbb{C}, \qquad \operatorname{Lie}(\T) = \mathbb C \begin{bmatrix}
		C & 0 \\
		0 & 0
	\end{bmatrix},
	\]
	and therefore \[
	\mathfrak{h} = \operatorname{Lie}(\U) \rtimes \operatorname{Lie}(\T) = \left\{ \begin{bmatrix}
		t C + s N & w \\
		0 & 0
	\end{bmatrix} \: \middle| \: w\in\mathbb C^4, t,s \in \mathbb{C} \right\}.
	\]
	Equivalently, \(\mathfrak h \cong W \rtimes_{C,N} (\mathbb C c \oplus \mathbb C n  )\), with \(W = \mathbb C^4\) abelian, \(CN = NC\), and
	\[
	[c,n] = 0, \qquad [c,w] = Cw, \quad [n,w] = Nw, \quad w \in W. 
	\]
	
	A direct computation gives
	\[
	Z(\mathfrak h) = 0, \qquad [\mathfrak h,\mathfrak h] = W.
	\]
	Thus \(\mathfrak h\) has no non-zero abelian summand.  
	Moreover, it is straightforward to show that every non-zero ideal of \(\mathfrak h\) meets \(W\) non-trivially.   
	
	\medskip
	
	\noindent \textit{Indecomposability of \(\mathfrak h\).} Aiming for a contradiction, assume 
	\[
	\mathfrak h = I \oplus J, \qquad [I,J] = 0,
	\] 
	for non-zero ideals \(I,J \lhd \mathfrak h\). Since \(\mathfrak h\) has no non-zero abelian summand, neither \(I\) nor \(J\) is contained in \(W\). Hence their images in the abelianization \(\mathfrak h / W \cong \mathbb C \tilde c \oplus \mathbb C \tilde n\) are non-zero. Choose \(x \in I\), \(y \in J\) such that
	\[
	x+W = \alpha \tilde{c} + \beta \tilde{n}, \qquad y + W = \gamma \tilde{c} + \delta \tilde{n}, \qquad \alpha \delta - \beta \gamma \neq 0.
	\]
	From \([I,J] = 0\), it follows that
	\[
	I \cap W \subset \operatorname{Ker}\left( \gamma C + \delta N \right), \qquad J \cap W \subset \operatorname{Ker}\left( \alpha C + \beta N \right).
	\]
	Since \(\alpha \delta - \beta \gamma \neq 0\), either \(\alpha \neq0 \) or \(\gamma\neq 0\).  If \(\alpha \neq 0\), then \(\alpha C + \beta N\) is invertible, so \(J \cap W = 0\), contradicting the fact that \(I,J\) meet \(W\) non-trivially. If \(\gamma\neq 0\), similarly \(I\cap W=0\), contradiction. Therefore \(\mathfrak h\) is indecomposable. \demo 
\end{example}

\begin{example}
	We give a non-nilpotent Heisenberg-by-cyclic example with non-trivial center and Zariski connected, \(\mathbb Q\)-indecomposable hull. It is similar to the previous example, but with \(\mathbb Z^4\) replaced by \(H_5(\mathbb{Z})\). 
	
Let \(W=\mathbb C^4=\operatorname{Span}_{\mathbb C}\{x_1,x_2,x_3,x_4\}\), and let \(x_1^*,x_2^*,x_3^*,x_4^*\) be the dual basis of \(W^*\). Equip \(W\) with the symplectic form
\[
\omega:=x_1^*\wedge x_2^*+x_3^*\wedge x_4^*
\in \bigwedge\nolimits^2 W^*.
\]	
	Define the rank-\(5\) Heisenberg group
	\[
	H_5(\mathbb Z)=\big\{(w,m/2) \, \big| \,  w\in \mathbb Z^4,\ m\in \mathbb Z\big\},
	\qquad
	(w,m/2)(w',m'/2)=\Big(w+w',\,\frac{m+m'}2+\frac12 \omega(w,w')\Big).
	\]
	Its center is \(Z(H_5(\mathbb Z))=\{(0,m/2)\mid m\in \mathbb Z\}\cong \mathbb Z\).
	
	Now define
	\begin{align*}
	D(x_1)&=5x_3,\qquad D(x_2)=-x_4,\qquad D(x_3)=x_1,\qquad D(x_4)=-5x_2, \\
		N(x_1)&=-5x_2,\qquad N(x_3)=x_4,\qquad N(x_2)=N(x_4)=0.
	\end{align*}
	Then \(D,N\in \mathfrak{sp}(W,\omega)\), \(D^2=5I_4\), \(N^2=0\), and \(DN=ND\). Set
	\[
	A:=9I_4+4D,
	\qquad
	M:=(I_4+N)A\in \operatorname{Sp}(W,\omega)(\mathbb Z).
	\]
	Since \(M\) preserves \(\omega\), the map
	\(
	\varphi(w,m/2)=(Mw,m/2)
	\) defines an automorphism of \(H_5(\mathbb Z)\). We define
	\[
	\Gamma:=H_5(\mathbb Z)\rtimes_\varphi \mathbb Z.
	\]
	Thus \(\Gamma\) is torsion-free polycyclic, non-nilpotent, and
	\[
	Z(\Gamma)=Z(H_5(\mathbb Z))\cong \mathbb Z.
	\]
	
	We now describe its \(\mathbb Q\)-algebraic hull. Let \(\HH_5\) be the \(5\)-dimensional Heisenberg \(\mathbb Q\)-group corresponding to \((W,\omega)\). Its Lie algebra is \(\mathfrak h_5 = W \oplus \mathbb Cz\) with bracket \[
	[w,w'] = \omega(w,w') z, \qquad w,w'\in W.
	\] Let
	\[
	\U:=\HH_5\rtimes \mathbb G_{\mathrm a},
	\qquad
	s\cdot (w,t)=\big((I_4+sN)w,t\big), \qquad \operatorname{Lie}(\U) = \mathfrak h_5 \rtimes_N \mathbb Cn.
	\]
	Further, let
	\[
	\T:=\overline{\langle A\rangle}^{\mathrm{Zar}}
	=
	\{\,aI_4+bD\in \operatorname{GL}_4(\mathbb C)\mid a^2-5b^2=1\,\}, \qquad \operatorname{Lie}(\T) = \mathbb C D,
	\]
	a \(1\)-dimensional \(\mathbb Q\)-torus, acting on \(\U\) by
	\[
	\tau\cdot (w,t,s)=(\tau w,t,s).
	\]
	Set
	\[
	\HH:=\U\rtimes \T.
	\]
	Then \(\HH\) is a connected solvable \(\mathbb Q\)-group with strong unipotent radical \(\U\). The group \(\Gamma = H_5(\mathbb Z) \rtimes \mathbb Z\) embeds into \(\HH(\mathbb Q)\) via the Heisenberg lattice \(H_5(\mathbb Z)\subset \HH_5(\mathbb Q)\) and the infinite cyclic generator
	\[
	g:=\exp(n)\cdot A \in \HH(\mathbb Q),
	\]
	whose conjugation action on \(H_5(\mathbb Z)\) corresponds to \(\varphi\). Moreover, \(\Gamma\) is Zariski dense and \(h(\Gamma) = \dim(\U) = 6\).  Hence \(\HH\) is the \(\mathbb Q\)-algebraic hull of \(\Gamma\).
	
	Finally, by computing 
	\[
	\mathfrak h:= \operatorname{Lie}(\HH) = \operatorname{Lie}(\U) \rtimes  \operatorname{Lie}(\T) \cong \mathfrak h_5 \rtimes_{N,D} \mathbb C^2,
	\]
	one deduces that \(\mathfrak h\) is an indecomposable Lie algebra. Hence, \(\HH\) is \(\mathbb Q\)-indecomposable. 
\end{example}

\section{Applications} \label{sec:applications}

As a consequence of our main result, Theorem~\ref{thm:CharacterizationAutomorphismDirectProductOfMinimaxGroups}, several properties of automorphisms and monomorphisms of a direct product can be reduced to the corresponding properties of the factors. For torsion-free nilpotent groups, this was already known implicitly, for instance in the study of expanding maps and Anosov diffeomorphisms. More precisely, \cite[Theorems~4.6 \&~4.8]{DW23} show that the existence of such maps on a direct product is equivalent to the existence on each factor. In this section, we derive applications to co-Hopfian groups and Reidemeister theory.

\subsection{Co-Hopfian groups}

Recall that a group \(\Gamma\) is called \emph{co-Hopfian} if every monomorphism \(\varphi\colon \Gamma \to \Gamma\) is an automorphism. In this subsection, we use Theorem~\ref{thm:CharacterizationAutomorphismDirectProductOfMinimaxGroups} to prove Theorem~\ref{thm:D-cohopfian} from the introduction. 
\begin{letterthmbody}
	Let $\Gamma=\Gamma_1 \times \cdots\times \Gamma_r$ be as in Theorem \ref{thm:CharacterizationAutomorphismDirectProductOfMinimaxGroups}. Then \(\Gamma\) is co-Hopfian if and only if every factor \(\Gamma_i\) is co-Hopfian.
\end{letterthmbody}

\begin{proof}
It is easy to see that the implication
\[\Gamma = \prod_{i=1}^r \Gamma_i \text{ is co-Hopfian} \implies \text{every \(\Gamma_i\) is co-Hopfian}\] 
 always holds, for any direct product of groups. It is only for the converse implication that we need to apply  Theorem~\ref{thm:CharacterizationAutomorphismDirectProductOfMinimaxGroups}.
 
 So assume that every factor \(\Gamma_i\) is co-Hopfian. Let \(\varphi\colon \Gamma \to \Gamma\) be a monomorphism of \(\Gamma\). We show that \(\varphi\) is an automorphism. Factorize \(\varphi = \theta \cdot \zeta\) as in Theorem~\ref{thm:CharacterizationAutomorphismDirectProductOfMinimaxGroups}, with associated permutation of the factors \(\sigma \in \mathcal S_r\).  Since \(\sigma^{r!} = \mathrm{Id}\), the iterate \(\theta^{r!}\) preserves each factor, so
 \[
 \theta^{r!} = \operatorname{Diag}(\theta_1, \theta_2,\ldots,\theta_r), \qquad \text{\(\theta_i\colon \Gamma_i \to \Gamma_i\) injective.}
 \]
By co-Hopfianity, each \(\theta_i\) is an automorphism, hence \(\theta^{r!}\) is an automorphism. Therefore \(\theta\) is surjective, so \(\theta\) is an automorphism. By Theorem~\ref{thm:CharacterizationAutomorphismDirectProductOfMinimaxGroups}, \(\varphi\) is an automorphism. 
\end{proof}

\subsection{Twisted conjugacy and Reidemeister numbers}\label{subsec:twistedConjugacy}

Given an endomorphism \(\varphi\) of a group \(G\), we define an equivalence relation on \(G\) by declaring \(x,y\in G\) to be \(\varphi\)-conjugate if \(x=zy\varphi(z)^{-1}\) for some \(z\in G\). Denote the resulting \(\varphi\)-conjugacy class by \([x]_{\varphi}\). Note that when \(\varphi = \mathrm{Id}\), we recover the usual conjugacy relation. Let \(\mathcal R[\varphi]\) be the set of all \(\varphi\)-conjugacy classes, and define the \emph{Reidemeister number}  of \(\varphi\) by \(R(\varphi):= \# \mathcal R[\varphi]\). Since \(G\) is usually countable, one writes \(R(\varphi)\in \mathbb N_{>0}\cup\{\infty\}\). The set \(\operatorname{Spec}_{\mathrm{R}}(G) = \{R(\varphi) \mid \varphi \in \operatorname{Aut}(G) \}\) is called the \emph{Reidemeister spectrum} of \(G\) and \(G\) is said to have the \emph{\(R_\infty\)-property} if \(\operatorname{Spec}_{\mathrm{R}}(G) = \{\infty\}\).  

The goal of this subsection is to prove Theorem~\ref{thm:E-Reidemeister} and Corollary~\ref{cor:F-spectrum} from the introduction.

\begin{letterthmbody}\label{thm:Reidemeister}
	Let $\Gamma=\Gamma_1\times\cdots\times \Gamma_r$ be as in Theorem~\ref{thm:CharacterizationAutomorphismDirectProductOfMinimaxGroups}.  
Let $\varphi\colon \Gamma\to\Gamma$ be a monomorphism.  	
Decompose $\varphi=\theta\cdot\zeta$ as in Theorem~\ref{thm:CharacterizationAutomorphismDirectProductOfMinimaxGroups}. 
Then
\(
R(\varphi)=R(\theta).
\) 
\end{letterthmbody}
The hypothesis that the identity components \(\HH^\circ_i\) are \(\mathbb Q\)-indecomposable is again essential in Theorem~\ref{thm:Reidemeister}. Indeed, for the directly indecomposable torsion-free polycyclic group \(\Gamma = \mathbb Z^4 \rtimes_{A} \mathbb Z\) from Example~\ref{example:Q-indecomposabilityPolycyclic}, there even exist automorphisms \(\varphi, \theta \in \operatorname{Aut}(\Gamma \times \Gamma)\) of the form described in Theorem~\ref{thm:CharacterizationAutomorphismDirectProductOfMinimaxGroups}, such that \(R(\varphi) \neq R(\theta)\) \emph{and} both are finite; see the next example.

\begin{example} Let \(\Gamma = \mathbb Z^4 \rtimes_A \mathbb Z\) be the directly indecomposable torsion-free polycyclic group from Example~\ref{example:Q-indecomposabilityPolycyclic}, with Zariski connected but \(\mathbb Q\)-decomposable hull \(\HH\). Recall that \(F = \operatorname{Fitt}(\Gamma) = \mathbb Z^4\).
	
	Define \(\theta_0\colon \Gamma \to \Gamma\) by \(\theta_0(v,k) = (Dv,-k)\), where 
	\[
	D = \begin{bmatrix}
		-1&-1&0&0\\
		2&1&0&0\\
		0&0&-1&1\\
		0&0&-2&1
	\end{bmatrix} \in \operatorname{GL}_4(\mathbb Z).
	\]
	Since \(A^{-1}D = DA\), we have that \(\theta_0\) is a homomorphism. Since \(\det D = 1\), it is an automorphism. As in Example~\ref{example:Q-indecomposabilityPolycyclic}, define central endomorphisms \(\zeta_i\colon \Gamma \to Z(\Gamma)\) by \(\zeta_i(v,k) = (L_iv,0)\), where
	\[
	L_1 = \begin{bmatrix}
		0&0&0&0\\
		2&2&-1&0\\
		0&0&0&0\\
		0&0&0&0
	\end{bmatrix}, \qquad 	L_2 = \begin{bmatrix}
		0&0&0&0\\
		-2&-2&1&0\\
		0&0&0&0\\
		0&0&0&0
	\end{bmatrix} = -L_1.
	\]
	Next, define endomorphisms of \(\Gamma \times \Gamma\):  
	\[
	\theta = \begin{bmatrix}
		\theta_0 & 0 \\
		0	& \theta_0
	\end{bmatrix}, \quad \zeta = \begin{bmatrix}
		0 & \zeta_1 \\
		\zeta_2 & 0
	\end{bmatrix}, \quad \varphi = \theta \cdot \zeta = \begin{bmatrix}
		\theta_0 & \zeta_1 \\
		\zeta_2 & \theta_0 
	\end{bmatrix}.
	\]

	\medskip 
	
	\noindent  \textit{Reidemeister number of \(\theta\).} The automorphism \(\theta_0\) preserves \(F\), and the induced maps are \(\theta_0'= D\) on \(F \cong \mathbb Z^4\) and \(\overline{\theta}_0 = -1\) on \(\Gamma /F = \mathbb Z\). Hence the Reidemeister classes induced on \(\mathbb Z\) are the cosets 
	\[
	\mathcal R \big[\overline{\theta}_0\big] = {\mathbb Z}/({(1 - (-1)) \mathbb Z} = \big\{  2 \mathbb Z, 1 + 2\mathbb Z \big\}.
	\]
	By the \emph{addition formula} \cite[Lemma 2.1]{GW03}, we have
	\[
	R(\theta_0) = \sum_{[\overline{\alpha}] \in \mathcal{R}[\overline{\theta}_0]} R\big(\tau_{\alpha} \circ \theta_0' \big) = R\big( \theta_0') + R(\tau_{t} \circ \theta_0' ) = R(D) + R(AD),
	\]
	where \(\tau_\alpha(v) = \alpha v \alpha^{-1}\) for \(v\in F\) denotes conjugation. Since \(F\) is abelian,
	\[
	R(D) + R(AD) = |\det(I_4 - D)| + |\det(I_4 - AD)| = 4+4 = 8.
	\]
	Therefore \(R(\theta) = R\big(\operatorname{Diag}(\theta_0, \theta_0)\big) = R(\theta_0)^2 = 64.\)
	\medskip
	
	\noindent \textit{Reidemeister number of \(\varphi\).} The Fitting subgroup of \(\Gamma \times \Gamma\) is \(F\times F\), and \(\varphi(F\times F) \leq F\times F\). The induced endomorphisms of \(F \times F\cong \mathbb Z^8\) and \((\Gamma \times \Gamma)/(F \times F) \cong \mathbb Z^2\) are 
	\[
	\varphi' = \varphi|_{F\times F} = \begin{bmatrix}
		D & L_1 \\ 
		L_2 & D
	\end{bmatrix}, \qquad  \overline{\varphi} = -I_2.
	\]
	Since \(\det(\varphi') = \det(\overline{\varphi}) = 1\), both are automorphisms, hence \(\varphi \in \operatorname{Aut}(\Gamma \times \Gamma)\). 
	
	On the abelian group \(\mathbb Z^2\), the Reidemeister classes of \(\overline{\varphi}\) are the four cosets
	\[
	\mathcal R\big[\overline{\varphi}\big] =	{\mathbb Z^2}/{\big(I_2 - (-I_2)\big) \mathbb Z^2} = \big\{ 2\mathbb Z^2, (1,0) + 2\mathbb Z^2, (0,1) + 2\mathbb Z^2, (1,1) + 2\mathbb Z^2 \big\}. 
	\]
	Again by the addition formula \cite[Lemma 2.1]{GW03}, we have
	\[
	R(\varphi) = \sum_{[\overline{\alpha}] \in \mathcal R[\overline{\varphi}]} R\big(\tau_{\alpha} \circ \varphi'\big).
	\]
	If \(\alpha = \big((v_1,k_1), (v_2,k_2)\big) \in \Gamma \times \Gamma\), then \(\tau_\alpha\) acts on \(F \times F\) as \(\operatorname{Diag}(A^{k_1},A^{k_2})\). We compute   
	\[
	R\left(\tau_{\big((v_1,k_1),(v_2,k_2)\big)} \circ \varphi'\right) = \left| \det \left( I_8 - \begin{bmatrix}
		A^{k_1} & 0 \\ 0 & A^{k_2}
	\end{bmatrix} \begin{bmatrix}
		D & L_1 \\ L_2 & D
	\end{bmatrix} \right) \right| = 32. 
	\]
	Thus \(R(\varphi) = 4 \cdot 32 = 128\). In particular, \(R(\varphi)\) and \(R(\theta)\) are both finite and \(R(\varphi)\neq R(\theta)\).  \demo 
\end{example}

\subsubsection{Computing \texorpdfstring{\(R(\theta)\)}{R(θ)} and the Reidemeister spectrum} 
 The Reidemeister number \(R(\theta)\) in Theorem \ref{thm:Reidemeister} may be computed using a product formula which applies to any such \emph{permutation-shaped endomorphisms} of a direct product of groups.  The general procedure is outlined below.

\medskip  

Let \(G=G_1\times G_2 \times \cdots \times G_r\) be any direct product of groups and fix \(\sigma \in \mathcal S_r\). Assume we are given (not necessarily injective) homomorphisms 
\[\phi_i\colon G_{\sigma^{-1}(i)} \to G_i \qquad \text{for \(i = 1,2,\ldots,r\)}, \] and let \(\phi\colon G \to G\) be the associated endomorphism permuting the factors:
\[
\phi(x_1, x_2,\ldots,x_r) = \big( \phi_1(x_{\sigma^{-1}(1)}),\phi_2(x_{\sigma^{-1}(2)}),\ldots,\phi_r(x_{\sigma^{-1}(r)}) \big).
\]
Let \(\sigma = C_1 \circ C_2 \circ \cdots \circ C_s\) be any disjoint cycle decomposition of \(\sigma\). For a cycle
\[
C_j = \begin{pmatrix}
	i_1 & i_2 & \cdots & i_m
\end{pmatrix},\quad \text{i.e. \(\sigma(i_{t}) = i_{t+1}\) and \(\sigma(i_m)=i_1\)},\]
 we define the endomorphism \(\phi_{C_j}\colon G_{i_1} \to G_{i_1}\) as the following composition: 
\[
G_{i_1} \overset{\phi_{i_2}}{\longrightarrow} G_{i_2} \overset{\phi_{i_3}}{\longrightarrow} G_{i_3} \overset{\phi_{i_4}}{\longrightarrow} \cdots \overset{\phi_{i_m}}{\longrightarrow} G_{i_m} \overset{\phi_{i_1}}{\longrightarrow} G_{i_1}.
\] 

\begin{proposition}\label{prop:productFormulaPermutationShapedEndomorphisms} With the above notation, \[R(\phi) = \prod_{j=1}^s  R(\phi_{C_j}).\] In particular, if \(\sigma = \mathrm{Id},\) then \(R(\phi) = \prod_{i=1}^r R(\phi_i)\). 
\end{proposition}

This proposition applies to the endomorphism \(\theta\colon \Gamma \to \Gamma\) in Theorem~\ref{thm:Reidemeister}, as it is given by 
\[
\theta(x_1, x_2,\ldots,x_r) = \big( \theta_1(x_{\sigma^{-1}(1)}),\theta_2(x_{\sigma^{-1}(2)}),\ldots,\theta_r(x_{\sigma^{-1}(r)}) \big)
\]
for some injective morphisms \(\theta_i\colon \Gamma_{\sigma^{-1}(i)}\to \Gamma_i\). 

\begin{proof}[Proof of Proposition \ref{prop:productFormulaPermutationShapedEndomorphisms}]
This is proved in \cite[Proposition~2.9]{Sen21} in the special case where \(G_1 = G_2 = \cdots = G_r\). The same telescoping argument applies here: 
one first decomposes $G=\prod_{j=1}^s \prod_{i\in C_j} G_i$ along the cycles $C_j$ of $\sigma$ and reduces to a single
cycle block. On one cycle block the computation is identical to \cite{Sen21}, since it only uses
cancellation in the twisted conjugacy relation and the composability of the maps along the cycle. We omit the routine details. 
\end{proof}

\begin{lettercorbody}
	\label{cor:ReidemeisterSpectrum}
	Let \(\Gamma_1, \ldots, \Gamma_k\) be \emph{non-isomorphic} non-abelian finitely generated virtually solvable minimax groups and let \(r_1, \ldots, r_k \geq 1\).  Assume that each \(\Gamma_i\) has a \(\mathbb Q\)-algebraic hull \(\HH_i\) with \(\mathbb Q\)-indecomposable identity component \(\HH_i^\circ\). Put \(\Gamma = \Gamma_1^{r_1} \times \cdots \times \Gamma_k^{r_k}\). Then  
	\[
	\mathrm{Spec}_{\mathrm{R}}(\Gamma)=\prod_{i=1}^k \left( \bigcup_{j = 1}^{r_i}  \mathrm{Spec}_{\mathrm{R}}(\Gamma_i)^{(j)} \right).
	\]
	In particular, \(\Gamma\) has the \(R_\infty\)-property if and only if  some factor \(\Gamma_i\) has the \(R_\infty\)-property. 
\end{lettercorbody}
\begin{proof}
	The proof is identical to that of \cite[Theorem 3.6]{Sen24}, with \cite[Proposition 3.4]{Sen24} replaced by Theorem~\ref{thm:Reidemeister}, using Proposition~\ref{prop:productFormulaPermutationShapedEndomorphisms}. We therefore omit the routine details.
\end{proof}

\subsubsection{Twisted conjugacy preliminaries}\label{subsubsec:tcp} The remainder of the current section is devoted to the proof of Theorem~\ref{thm:Reidemeister}. First, we introduce the necessary technical preliminaries on twisted conjugacy. We follow the notation of \cite[\S1.2]{Sen23}, where all details are carefully worked out.

Let \(\varphi\) be an endomorphism of a group \(G\). For \(g \in G\), we define the \(\varphi\)\emph{-stabilizer} of \(g\) by 
\[
	\operatorname{Stab}_{\varphi}(g) := \{ x \in G \mid x g \varphi(x)^{-1} = g\}.
\]
Let \(C \leq Z(G)\) be a central subgroup such that \(\varphi(C) \leq C\). Let \(\pi\colon G \to G/C\) denote the canonical projection, and let \(\varphi|_{C}\) and \(\overline{\varphi}\) denote the induced endomorphisms of \(C\) and \(G/C.\)  
First we state a fact which holds in general.

\begin{lemma} \label{lemma:ReidemeisterNumberOfQuotientIsSmallerThanReidemeisterNumber}
 With the above notation, 
 \[R\big(\overline{\varphi}\big) \leq R(\varphi). \]
\end{lemma}
\begin{proof}
The map
 \(\mathcal R[\varphi] \to \mathcal R[\overline{\varphi}]\colon [x]_{\varphi} \mapsto [xC]_{\overline \varphi}\)
 is a well-defined surjection.
\end{proof}

\begin{lemma}[{See \cite[\S1.2]{Sen23}}]\label{lemma:twistedConjugacySumFormula}
	With the above notation, 
\begin{enumerate}[label=\emph{(\arabic*)}, ref={(\arabic*)}]
	\item For each \(g \in G\), the set \[
	\Big[g,\pi^{-1}\big(\operatorname{Stab}_{\overline{\varphi}}(gC)\big)\Big]^{\varphi} := \Big\{ g^{-1} x^{-1} g \varphi(x) \, \Big| \, x \in \pi^{-1}\big(\operatorname{Stab}_{\overline{\varphi}}(gC)\big) \Big\}
	\]
	is a subgroup of \(C\), \label{item:lemmatwistedConjugacySumFormulai}
	\item For each \(g \in G\), the map 
	\begin{align*}
	\lambda_g^{\varphi} \colon 	\Big[g,\pi^{-1}\big(\operatorname{Stab}_{\overline{\varphi}}(gC)\big)\Big]^{\varphi} \times \mathcal R\big[\varphi|_{C}\big] &\longrightarrow \mathcal R\big[\varphi|_{C}\big] \\
\big( x, [ z ]_{\varphi|_C} \big)	&\longmapsto [ xz ]_{\varphi|_C}
	\end{align*}
	is a well-defined left action of \(\Big[g,\pi^{-1}\big(\operatorname{Stab}_{\overline{\varphi}}(gC)\big)\Big]^{\varphi}\) on  \(\mathcal R\big[\varphi|_{C}\big]\), and  \label{item:lemmatwistedConjugacySumFormulaii}
	\item The following sum formula holds: \[R(\varphi) = \sum_{[gC] \in \mathcal R[\overline{\varphi}]} \# \Big\{ \mathrm{orbits}\: \mathrm{of}\:\lambda_g^{\varphi} \Big\}. \] \label{item:lemmatwistedConjugacySumFormulaiii}
\end{enumerate}
\end{lemma}

The following result may be found implicitly in \cite[Proof of Lemma 4]{Jab08}. 
For the full statement, along with a detailed proof, see \cite[Proposition 1.2.11]{Sen23}.
\begin{lemma}\label{lemma:finitelygeneratedresiduallyfinitegroupWithFiniteReidemeisterNumberHasFiniteStabilizers}
	Let \(G\) be a finitely generated residually finite group. Let \(\varphi\colon G\to G\) be an endomorphism. If \(R(\varphi) < \infty\), then \(\operatorname{Stab}_{\varphi}(g)\) is finite for every \(g\in G\).
\end{lemma}

\subsubsection{Proof of Theorem \texorpdfstring{\ref{thm:Reidemeister}}{E}} 

For a finitely generated torsion-free nilpotent group, Senden's arguments in \cite[\S3]{Sen24} use the \emph{product formula} \cite[Proposition 2.6]{Sen24},
\[R(\varphi) = R(\varphi|_{Z(\Gamma)})\cdot R(\overline{\varphi}),\]
where \(\overline{\varphi}\) is the induced automorphism of \(\Gamma/Z(\Gamma)\); see also \cite[Corollary 1.2.13]{Sen23}. However, this formula relies on the fact that \(\Gamma/Z(\Gamma)\) is torsion-free, which need not hold in our case:
\begin{example}
Consider the torsion-free polycyclic group
\[
\Gamma := \mathbb Z^2 \rtimes_{-I_2} \mathbb Z = \big\langle a,b,t \: \big| \: [a,b] = 1, tat^{-1} = a^{-1}, tbt^{-1} = b^{-1} \big\rangle.
\]
We can view the elements of \(\Gamma\) as pairs \((v,k)\) where \(v \in \mathbb Z^2\) and \(k \in \mathbb Z\). Multiplication is given by
\[
(v,k) \cdot (w,\ell) = \left (v + (-1)^k w, k + \ell \right).
\]
Routine calculations show that \(Z(\Gamma) = \langle t^2 \rangle = \{(0,k) \mid k\equiv 0 \bmod 2\}\) and so \(\Gamma / Z(\Gamma) = \mathbb Z^2 \rtimes_{-I_2} \mathbb Z_2\) is not torsion-free. 
Consider the automorphism \(\varphi\colon \Gamma \to \Gamma\) defined by \[
M = \begin{bmatrix} 0 & 1 \\ -1 & 0 \end{bmatrix} \in \operatorname{GL}_2(\mathbb Z), \qquad \varphi(v,k) = (Mv,-k).
\]
Straightforward calculations yield that \(R(\varphi|_{Z(\Gamma)}) = 2\),  \(R(\overline{\varphi}) = 4\) and \(R(\varphi) = 4\), therefore  the product formula fails: 
\[
R(\varphi) \neq R(\varphi|_{Z(\Gamma)}) \cdot R(\overline{\varphi}).
\]
\demo  \end{example}

We have the following more general proposition, which states that under certain conditions adding \emph{central noise} to an endomorphism does not change the Reidemeister number. 
After proving this proposition, we show that our specific class of groups and morphisms \(\varphi = \theta \cdot \zeta\) in Theorem~\ref{thm:CharacterizationAutomorphismDirectProductOfMinimaxGroups} satisfies its assumptions. 
\begin{proposition}\label{prop:GeneralTheoremAboutReidemeisterEquality}
Let \(G\) be a group with \(Z(G)\) torsion-free and \(G / Z(G)\) finitely generated and residually finite.  Assume that \(\varphi = \theta \cdot \zeta\) holds for some endomorphisms \(\varphi, \theta,\zeta\colon G\to G\) satisfying \(\theta(Z(G)) \leq Z(G)\), \(\zeta(G)\leq Z(G)\), \(\zeta|_{Z(G)} = 1\). Then \(R(\varphi) = R(\theta)\).
\end{proposition}
\begin{proof}
We apply our \emph{Twisted conjugacy preliminaries} (\ref{subsubsec:tcp}) to both \(\varphi\) and \(\theta\), with \(C:= Z(G)\) the full center. As \(\zeta|_{Z(G)} = 1\) and \(\varphi = \theta \cdot \zeta\), we see that \(\varphi\) and \(\theta\) induce the same maps on \(Z(G)\) and on the quotient \(G/Z(G)\): 
\[
\varphi|_{Z(G)} = \theta|_{Z(G)} \quad \text{and} \quad \overline{\varphi} = \overline{\theta}.
\]
In particular, the induced twisted conjugacy relations on \(Z(G)\) and \(G/Z(G)\) are identical. 

If \(R(\overline{\varphi}) = R(\overline{\theta}) = \infty\), then \(R({\varphi}) = \infty\) and \(R({\theta}) = \infty\) by Lemma~\ref{lemma:ReidemeisterNumberOfQuotientIsSmallerThanReidemeisterNumber}. Hence the formula \(R(\varphi) = R(\theta)\) holds trivially. So we may assume that \(R(\overline{\theta}) = R(\overline{\varphi}) < \infty\). 
Since \(G/Z(G)\) is finitely generated and residually finite, Lemma~\ref{lemma:finitelygeneratedresiduallyfinitegroupWithFiniteReidemeisterNumberHasFiniteStabilizers} applies to \(\overline{\varphi}\): every \(\overline{\varphi}\)-stabilizer \(\operatorname{Stab}_{\overline{\varphi}}\big(gZ(G)\big) \) is finite. 
	
	Now consider the sum formula in Lemma~\ref{lemma:twistedConjugacySumFormula}\ref{item:lemmatwistedConjugacySumFormulaiii} for both \(\varphi\) and \(\theta\). It computes \(R(\varphi)\) (resp.\ \(R(\theta)\)) as a sum over all \(\big[gZ(G)\big] \in \mathcal R[\overline{\varphi}] = \mathcal R[\overline{\theta}]\) of the number of orbits of an action \(\lambda_g^{\varphi}\) (resp.\ \(\lambda_g^{\theta}\)) on \(\mathcal R[\varphi|_{Z(G)}] = \mathcal R[\theta|_{Z(G)}]\). It thus suffices to prove that the actions \(\lambda_g^{\varphi}\) and \(\lambda_g^{\theta}\) are identical for every \(g \in G\). 
	
	For a fixed \(g \in G\), recall how \(\lambda_g^\varphi\) and \(\lambda_g^\theta\) are defined in Lemma~\ref{lemma:twistedConjugacySumFormula}\ref{item:lemmatwistedConjugacySumFormulaii}. We see that they are defined by the same formula, and that they act on the same set \(\mathcal R[\varphi|_{Z(G)}] = \mathcal R[\theta|_{Z(G)}]\). Hence, it suffices to show that they act by the same group, i.e. that \(\varphi\) and \(\theta\) coincide on the subgroup
	\[H:= \pi^{-1}\Big(\operatorname{Stab}_{\overline{\varphi}}\big(gZ(G)\big)\Big) =  \pi^{-1}\Big(\operatorname{Stab}_{\overline{\theta}}\big(gZ(G)\big)\Big) \leq G.\] As \(\varphi = \theta \cdot \zeta\), it suffices to show that  \(\zeta(H) = 1.\) Since \(\zeta|_H\colon H \to Z(G)\) factors through the finite group \(H/Z(G) \cong \operatorname{Stab}_{\overline{\varphi}}\big(gZ(G)\big),\) the image \(\zeta(H)\) is a finite subgroup of \(Z(G)\), which is torsion-free. Therefore,  \(\zeta(H) = 1\) and this concludes the proof.
\end{proof}
To prove Theorem \ref{thm:Reidemeister}, it now suffices to show that the groups and endomorphisms under consideration satisfy the assumptions of Proposition~\ref{prop:GeneralTheoremAboutReidemeisterEquality}.

  The hardest part of this verification is encapsulated by the following lemma.

\begin{lemma}\label{lemma:AssumptionsOfGeneralTheoremReidemeisterProduct}
	Let \(\Gamma\) be a finitely generated virtually solvable minimax group. Then  \(\Gamma/Z(\Gamma)\) is also a finitely generated virtually solvable minimax group.  Moreover, if \(\Gamma\) is \emph{\textsf{WTN}}, then so is \(\Gamma/Z(\Gamma)\).
\end{lemma}
\begin{proof}
	 A quotient of a group with \textbf{Max} (resp.\ \textbf{Min}) again satisfies \textbf{Max} (resp.\ \textbf{Min}). Hence a quotient of a minimax group is again minimax. The same is true for finite generation and virtual solvability. 
 Hence it suffices to prove the \emph{Moreover, ...} statement. 
 
 \medskip 
	
 So assume that \(\Gamma\) is \textsf{WTN}. By Proposition~\ref{prop:WTNiffWFN}, it suffices to show that \(\Gamma/Z(\Gamma)\) is \textsf{WFN}, i.e. that \(\Gamma/Z(\Gamma)\) has no non-trivial finite normal subgroups. 
 
 Let \(\overline{N} \lhd \Gamma/Z(\Gamma)\) be finite, and let \(N\) be its full preimage in \(\Gamma\). Since \(Z(\Gamma) \leq Z(N)\), the quotient \(N / Z(N)\) is finite, hence Schur's theorem (\cite[10.1.4]{Rob96}) yields that \(N'=[N,N]\) is finite. As \(N'\) is characteristic in \(N\) and \(N \lhd \Gamma\), we have \(N' \lhd \Gamma\); since \(\Gamma\) is \textsf{WTN}, \(N' = 1\), so \(N\) is abelian.  Similarly,  \(\tau(N)\) is characteristic in \(N\) and \(N \lhd \Gamma\), hence \(\tau(N) = 1\) by \textsf{WTN}. Therefore, \(N\) is torsion-free abelian.  
 
 Finally, \(Z(\Gamma)\) has finite index in \(N\), and conjugation by any \(g \in \Gamma\) induces an automorphism of \(N\) that fixes \(Z(\Gamma)\) pointwise. Since \(N\) is torsion-free abelian, this automorphism is the identity, hence \(N \leq Z(\Gamma)\). Therefore \(N = Z(\Gamma)\) and \(\overline{N} = 1\).
\end{proof}

\begin{proof}[Proof of Theorem \ref{thm:Reidemeister}]
	Let \(\Gamma = \Gamma_1 \times \cdots \times \Gamma_r\) be as in Theorem~\ref{thm:CharacterizationAutomorphismDirectProductOfMinimaxGroups}. Let \(\varphi\) be a monomorphism of \(\Gamma\), and factorize \(\varphi = \theta \cdot \zeta\) as in Theorem~\ref{thm:CharacterizationAutomorphismDirectProductOfMinimaxGroups}. It suffices to verify that \(\Gamma\) and the maps \(\varphi,\theta,\zeta\) satisfy the assumptions of Proposition~\ref{prop:GeneralTheoremAboutReidemeisterEquality}.
	\begin{itemize}[leftmargin=1.8em]
\item	\emph{\(Z(\Gamma)\) is torsion-free}: This holds because a non-trivial torsion element of \(Z(\Gamma)\) would generate a non-trivial finite normal subgroup of \(\Gamma\), yet \(\Gamma\) is \textsf{WTN}.

\item	\emph{\(\Gamma / Z(\Gamma)\) is finitely generated and residually finite}: By  Lemma~\ref{lemma:AssumptionsOfGeneralTheoremReidemeisterProduct}, \(\Gamma / Z(\Gamma)\) is again a finitely generated virtually solvable minimax group without torsion normal subgroups. By \cite[Lemma 2.6]{DP26}, it follows that \(\Gamma / Z(\Gamma)\) is residually finite.
	
\item	{\(\theta(Z(\Gamma)) \leq Z(\Gamma)\)}: It holds that \(Z(\Gamma) = Z(\HH) \cap \Gamma \leq Z(\HH)\). Since the induced \(\mathbb Q\)-automorphism \(\Theta\)  satisfies \(\Theta (Z(\HH)) \leq Z(\HH)\), we deduce that \(\theta(Z(\Gamma)) \leq Z(\Gamma)\).
	
\item	\(\zeta|_{Z(\Gamma)} = 1\): Let \(x \in Z(\Gamma)\). By the previous bullet point, \(Z(\Gamma)\leq Z(\HH)\), hence \(x \in Z(\HH)\).  Under the assumptions of Theorem~\ref{thm:CharacterizationAutomorphismDirectProductOfMinimaxGroups}, the \(\mathbb Q\)-group \(\HH\) admits no abelian \(\mathbb Q\)-factor. Thus, Proposition~\ref{prop:K-indecomposableGroupsIFFK-algebraicallyIndecomposableLiealgebras} implies that \(Z(\HH) = Z(\HH)^\circ \subset [\HH, \HH]\). Finally, by extending \(\zeta\) to a central \(\mathbb Q\)-endomorphism \(\Psi \colon \HH \to \HH\) as in the proof of Proposition~\ref{prop:virtuallypolycyclicGroupsInjectiveIFFinjectiveandAutomorphismIFFAutomorphism}, we  conclude that \[\zeta(x) = \Psi (x) \in \Psi\big([\HH,\HH]\big) = \{e\}.\] \end{itemize} 
Therefore all assumptions of Proposition \ref{prop:GeneralTheoremAboutReidemeisterEquality} are satisfied, and hence \(R(\varphi)=R(\theta)\).
\end{proof}

\section*{Acknowledgements} We thank \textbf{Lukas Vander Stricht} for his master's thesis \cite{VdS24} (supervised by the first author), which worked out the torsion-free nilpotent case in full detail and explored initial steps towards the torsion-free polycyclic setting via \(\mathbb Q\)-algebraic hulls. In particular, the converse direction in Example~\ref{example:Q-indecomposableNILPOTENT} was originally found by him.

\section*{Statements and Declarations}

K.V.'s research is supported by Methusalem grant METH/21/03 --- long-term structural funding of the Flemish government. The authors have no conflict of interest to declare that are relevant to this article. Data sharing not applicable to this article as no datasets were generated or analyzed during the current study.

\end{document}